\documentclass[12pt]{article}

\usepackage{eufrak}
\usepackage{amsfonts}
\usepackage{amsmath}
\usepackage{amssymb}
\usepackage{amsthm}

\usepackage[makeroom]{cancel}

\usepackage{tikz}
\usetikzlibrary{matrix}

\usepackage[titletoc]{appendix}
\usepackage{titlesec}

\def\Ph1{P^{(h_1)}}
\def\Ph2{P^{(h_2)}}

\def\b0{{\bf 0}}

\def\bnu{\bar\nu}

\def\1lbnu{{\bnu}_{\lambda_1}}
\def\2lbnu{{\bnu}_{\lambda_2}}

\newtheorem{thm}{Theorem}[section]

\newtheorem{lem}[thm]{Lemma}
\newtheorem{cor}[thm]{Corollary}

\newtheorem{ex}{Example}

\theoremstyle{plain}

\newtheorem{defn}[thm]{Definition}

\newtheorem{rem}[thm]{Remark}

\newtheorem{incsol}[thm]{Incorrect Solution}

\newtheorem{impsol}[thm]{Improved Solution}

\newtheorem{compsol}[thm]{(Complete) Solution}

\titleformat{\section}{\bfseries}{\thesection.}{0.5em}{}
\titleformat{\subsection}{\normalfont\itshape}{\thesubsection.}{0.5em}{}
\titleformat{\subsubsection}{\normalfont\itshape}{\thesubsubsection.}{0.6em}{}

\begin{document}

\title{Probability and Hilbert's VI problem.}

\author{ A. Gandolfi\footnote{ NYU Abu Dhabi
and Dipartimento di Matematica e Informatica U. Dini, Universit\`a di Firenze, Firenze} \\
{\footnotesize email: ag189@nyu.edu  - 
gandolfi@math.unifi.it}
}


\date{}
\maketitle

\begin{abstract}

This work has been prompted by the surprising lack of mathematical 
coherence in the common usage of some of the fundamental entities in
the theory of probability, with an inherent risk of contradiction.
While disentangling the intricacies, we realized that the same
issue has been raised many times, with only partial solutions, notably
by Boole, Hilbert, De Finetti and Renyi, among others. In particular,
a restoration of foundational coherence in the usage of probability theory 
 appears to be a missing piece in the solution of
Hilbert VI problem.

Here we solve the problem by  a new formalization of probability theory
based on a  minimal collection of axioms with additional 
context dependent  conditions, whose overall consistency is then
semantically verified.
In Elementary Probability, i.e. probabilities involving boolean combinations of
 finitely many events, our theory leads to algebraization and, using Tarski Seidenberg
 reduction, to a proof of decidability of all problems. 
 Inconsistency in Elementary Probability, on the other hand, is equivalent to, 
suitably redefined, arbitrage or Dutch Book. In the continuous case this leads
to nonstandard analysis.

\end{abstract}
{\it Key words and phrases:} probability, discrete probabilities, conditional probabilities,
independence, moment problems, finite additivity, existence theorems, Hilbert, Boole, De Finetti,
Kolmogorov, axioms,  model theory, consistency, elementary probability,
Tarski Seidenberg, positivstellensatz, Dutch Books, joint normals. \\
{\it AMS subject classification:} 60C05, 60K35.

\newpage

\tableofcontents

\section{Introduction}\label{intro}

There is a long history in the search of a theory of probabilties
 see e.g. \cite{TF1973, VP94}, and Hilbert VI problem
\cite{LC} calls for its axiomatization; this has been generally interpreted,
most of the time by Hilbert himself, as the quest for a collection of few axioms from 
which the rest of the theory can be derived. Kolmogorov \cite{K}
proposed one such axiom system which, although disputed by De Finetti and others, has lead
to a clarification of the foundations, and has become the standard accepted solution.

Yet, there are aspects which, surprisingly, have been  mostly overlooked to this day. They concern a
 lack of mathematical coherence in most of the 
applications and exercises involving probability theory. Fundamental concepts, among which
independence and conditional probabilities, are presented and used in two inconsistent
ways: in the theory, they are introduced as definitions, but in the applications
they are unfoundedly taken as assumptions. Some confusion about the role of the main probabilistic concepts was 
 recognized by Hilbert, who in 1905
 indicates that 
"at its present state of development, the "axioms" and the "definitions" somewhat overlap with each other" \cite{LC, UK2011, DH1}. 
The "overlap" has never ceased in applications of probability.
In addition, this confusion  spawns 
a potential risk of contradiction, as illustrated, for instance, by the exercise in Appendix
\ref{Contradictory}; the consequences of inconsistency could obviously be quite severe in applied
contexts, as unwarranted conclusions, for instance about safety,
 might be drawn from contradictory assumptions. 
 
 There have been several other calls and attempts at formalizing probability
theory, notably by Bohlmann \cite{GB}, Keynes \cite{Ke},
 Savage, Popper \cite{Po},
Renyi \cite{R, K84}, and theories like Quantum Probability \cite{RS2007, Pit1989b},
Free Probability \cite{B03} and Bayesian Probability, raising diverse issues such as
 the  use
of sets in Kolmogorov
axioms, the significance of countable additivity, and again the role
of independence and conditional probabilities; none of these seems to be
complete satisfactory. In parallel, the problem  of
potential contradictions is 
 explicitly mentioned
 in the works of Boole \cite{B54, Ha2} and De Finetti \cite{DF2, DF3}, and later in
PSAT \cite{ N, HJ}; but these last researches consider only  the linear cases,  
and hence cannot deal with concepts like
 independence.  To conclude, several paradoxical statements have also been proposed 
\cite{E2012, Ha, Lyon(2014)}, often intertwined with the same lack
of coherent usage of the basic concepts. The need of a formulation which
 is able to deal with  possible inconsistencies and  other issues 
 seems then
 to still be a missing piece in the solution of Hilbert VI problem.

The possibility of inconsistencies casts a different light into the quest
for axiomatization of probability. In fact, a new system of axioms is created every
time a new problem is considered, but many problems in the analysis of random phenomena
are
so immediate  that the need of a
consistency check seems to be missed during the mathematical formalization. In addition,
the foundations of probability theory proposed so far, and, even more, the overall idea
that axiomatization is aimed at finding a small collection of far reaching axioms,
offer no tool to prevent inconsistencies in applied problems. Indeed, no axiomatization prevented
the exercise reported
in Appendix \ref{Contradictory}  from being considered appropriate.

We seem, therefore, compelled to assign an additional task to the 
axiomatization of probability; in essence, we need a flexible system 
which is able to adapt to single problems, indicating both how
to prevent inconsistencies and how to preserve the calculative power of 
probabilistic concepts. This is problem we treat in this paper by
proposing a new formalization.
We see below that in such formalization concepts like independence and conditional probabilities
end up consistently playing a dual role, acting both as
 assumptions, whose consistency has to be checked, 
and  as definitions, which are the starting points of calculations.
As a matter of fact, also additivity is revealed to posses the same type of duality.

At first, the idea was a semantical consistency check: once the hypothesis of a
problem have been identified, one has to look for a probability space
satisfying all the hypothesis, showing thereby a relative consistency
(absolute consistency is essentially ruled out by G\"odel's
second Incompleteness Theorem \cite{G}). This is
the procedure suggested by Model Theory, also at the basis of
moment problems and PSAT. In Section \ref{Elementary probability}
we develop this direction 
by introducing an algebraization of Elementary Probability
which ultimately leads to show its decidability
 (a result which seems to  fulfill Boole's original claim of having 
a way of solving "all problems in probability"  \cite{Ha2}). 
As PSAT is a special case of the algebraic problem we formulate in Elementary
Probability, which could be 
called PPSAT (Polynomial PSAT), this too in NP-complete.

We realized, however, that, in pursuing the above direction, the specific assumptions of each
problem and the usual axioms of a probability space end up being treated in the same way
(we then name them all "requirements").
This offers  the chance to relax the standard axioms, allowing
parts of them to become context dependent. This is done in the paper
by starting from very  basic probability spaces (related to plausibilities in quantum context), and then introducing
the notion of "jointly perceivable" events, a notion whose treatment ends up paralleling 
 that of the standard collective independence. Once the two notions are employed together
 one can give a coherent foundation to diverse formulations of probability,
 each one being identified by some requirements which are constantly
 taken within that formulation, with additional ad hoc requirements
 in each problem. This is described in Section 
 \ref{Probability}. Section \ref{Elementary probability}  then
goes back to Elementary Probability and its algebraization.
 
 \bigskip
 
 In Elementary Probability, we see that if there is no model satisfying
 all the requirements of a problem, then one can 
 determine a suitably redefined arbitrage mechanisms, or Dutch Book
 (see Section \ref{PSDB}).
This generalizes the foundational work of De Finetti,
and the  Fundamental Theorem of Asset Pricing \cite{DS08}.
The construction is based on Stengle's Positivstellensatz \cite{LPR},
and shows that the assumptions of a problem are consistent if and only if, 
provided some replicability of the events,
it is not possible to extract a sure profit from a believer 
of those requirements. 
Outside of Elementary Probability the Dutch Book method
encounters some difficulties, as its absence  is no longer equivalent
to existence of a model, see Section \ref{MP}; 
 this phenomenon is known in other contexts,
 and seems to require either nonstandard analysis \cite{HL85}
 or extensions of the concept of arbitrage \cite{DS08}.

\bigskip

Our proposed method entails several questions about logic. 
Following Model Theory  \cite{TZ, E}, we need to identify a formal language, a class of structures
and correspondence rules; in addition 
a truth predicate \cite{T44} would be needed to ascertain satisfiability.
 As  there  does not seem to be
an optimal  choice for the language (see, e.g.
\cite{Va}), it appears more reasonable that in our context the language 
  itself is chosen in relation to the requirements,
  allowing the flexibility of selecting a rich model and proof theory for simple 
  problems, and a more expressive language
  for more elaborate ones. We do not pursue these considerations further in the
  present paper.

\bigskip

Summarizing, our proposal, which is to a large
 extent just a formalization of commonly used procedures, is that
 the mathematical analysis of probabilities
 should be reversed:
 instead of looking for axioms which capture as many 
 situations as possible, one can (quite freely) select a collection of
 assumptions (i.e. axioms) for each specific problem
 and then derive consequences from there, with the sole additional 
 constraint of a preliminary consistency check  (via existence of a model).

 Notice that, along the way of our formalization, we also forgo the need of
 having a preliminarily  fixed set, a desideratum which has
 been raised by several authors such as Keynes \cite{Ke}
or Popper \cite{Po}, and, in some form, by Tao's ansatz 
  \cite{T}. In addition, although we do not present the details here, 
  it is clear that our treatment allows to make a parallel development
  of various formulations of probability theory, and also of
  some
  theories which are close to that of probability, such as
  Choquet's Capacity
or Shafer's Evidence \cite{D}.  

\vskip 1truecm
 
Throughout the paper boldface symbols such as
${\bf x} = (x_1, x_2, \dots, x_k)$
 indicate vectors whose coordinates are  clear from the context;
$\delta_A$ is the Kronecker delta function of  $A$.

On first reading, it is possible to focus on Elementary Probability
by going directly to Appendix \ref{Contradictory} and
Section  \ref{Elementary probability}.

\section{Probability}\label{Probability}

\subsection{Requirements}

All requirements will be set on equal footing, but we single out
a minimal collection which serves as a basis for the entire theory.

\begin{defn}
A  {\bf  basic probability space} is  a triple $(\overline \Omega, 
\overline{\mathbb A}, \overline P)$ where $\overline \Omega$ is a set,
$\overline{\mathbb A}$ is a family of subsets of $\overline \Omega$
containing $\overline \Omega$ and $\emptyset$,
and $\overline P$ is a real valued function on $\overline{\mathbb A}$ such that 
 \begin{itemize}
\item[(a)] $\overline P(\emptyset)= 0 $;
\item [(b)]  $\overline P(\overline \Omega)=1 $;
\item[(c)] for every $\overline A \subseteq \overline B \subseteq \overline \Omega$,
$\overline P(\overline A) \leq \overline P(\overline B)$.
\end{itemize}
Elements $\overline A$ of $\overline{\mathbb A}$ are called basic events,
and $\overline P $ is called basic probability.
\end{defn}
That these assumptions are not contradictory can be seen with 
$\overline P(\overline A)=\delta_{\overline A,\overline \Omega}$
on any set $\overline \Omega$.

Note that basic probabilities appear as "plausibilities" in quantum contexts
\cite{F86, FL15}.
Note also  that we decorate the symbols with a hat as they represent
"concrete" structures, i.e. sets.

 \medskip
 
 Standard axiomatizations of Probability Theory identify more axioms, but, as mentioned, we are incorporating any further  
 assumption with the  specific case by case ones; we name them all
 requirements.

\begin{defn} \label{Requirement}
A   {\bf requirement} is any statement which can hold for a basic probability space.
\end{defn}
Initial examples of requirements are  $P(A)=1/2$ or
 there are finitely many events; later, when the theory is developed,
 requirements take more elaborate forms like constraints on moments, 
 a random variable being a martingale, or a stochastic process
 satisfying a SPDE.
  Notice that here we use symbols without bar, to express the fact
that  requirements are stated before  specific basic probability space 
or random variables are determined.

The interest is in collection of requirements:
\begin{defn} \label{FirstDefinPreEnvir}
A   {\bf  probability pre-environment } is a quadruple \\
$(( \Omega, 
{\mathbb A},  \mathcal P), \mathcal R)$, in which $\Omega$ is a symbol;
$\mathcal P$ is a set of  symbols containing at least $P$; 
$\mathbb A$ is a set of symbols containing at least 
$\emptyset$ and $\Omega$; $\mathcal R$ 
is a collection of requirements about the symbols in $( \Omega, 
{\mathbb A},  \mathcal P)$.
\end{defn}

Examples of probability pre-environments appear
everywhere
in the usual development of probability theory,
both at abstract levels as assumptions of a theorem, and
in problems as collections of hypothesis.

In fact, once a pre-environment is described, consequences can
be derived by a deductive calculus. This is the usual modus operandi
both for theoretical developments and for applications of 
probability theory.

Requirements are then stratified, in the sense that once deductions
are drawn from some requirements, further concepts can be determined
which become the basis of new requirements. For instance, 
one typical requirement is that $\mathbb A$ is a $\sigma$-algebra,
that there is a function $X:\Omega \rightarrow \mathbb R$ measurable
with respect to $\mathbb A$; if some additivity is required for $P$ 
 then one can define integration with respect to $P$,
and then require certain properties for the moments of $X$. 
A similar process takes place with independence (see also below).
In practice, the introduction of  of requirements and
pre-environments can be seen as a merely terminological clarification of
the standard probability theory.

Our new formalization points, however, directly to the fact
that all the deductive effort could be groundless if a contradiction
is present among the requirements. The next section gives a semantic
interpretation which closes the
circle of our definitions and insures consistency; later on we discuss
what can happen with inconsistency.

\bigskip
\subsection{Probability environments} \label{Probability environments}

The following definition provides at the same a more meaningful constraint on the type of
requirements which can appear in a probability pre-environment, and
model theoretical consistency.
\begin{defn} \label{FirstDefinEnvir}
A   {\bf  probability environment } is a probability \\
pre-environment $(( \Omega, 
{\mathbb A},  \mathcal P), \mathcal R)$
such that there exists a basic probability space
$(\overline \Omega, 
\overline{\mathcal A}, \overline P)$  satisfying all the requirements.
\end{defn}
A more precise description of how the constraint
are to be satisfied involves an interpretation of the
symbols in the pre-environment in terms of elements of the basic probability
space. This depends on the type of logic; a schematic description is as follows.
 First, $\overline{\mathcal A}$ contains one element for each member of $\mathbb A$; next,
if,  in each of the requirements in $\mathcal R$,
$\overline \Omega$ replaces $\Omega$, 
the corresponding members in $\overline{\mathcal A}$  replace those in $\mathbb A$,
 $\overline P$ replaces $P$,
and each other symbol in $\mathcal P$ is replaced by that of a mathematical entity defined in terms of
$(\overline \Omega, 
\overline{\mathcal A}, \overline P)$, then
the requirements in $ \mathcal R$ hold.
In such case, the elements of ${\mathbb A}$ are
 events, $P$ is a probability,  the basic probability space is called a 
(probability) model
for the environment. We indicate by $\Psi$ a map which realizes the
above correspondence.

Notice that in the above definition
events are not sets, the probability is not a function etc.

When consistency of a probability environment is ascertained, then
consequences can be consistently derived  by inference rules. 
The model theoretical determination of consistency introduces also 
the possibility of a semantic sequent calculus. 
We say that a statement is a {\bf possible consequence} of
a probabilistic environment   if the 
statement holds for  at least one of the basic probability spaces satisfying the requirements.
A statement is a {\bf necessary consequence} if 
it holds for all the basic probability spaces satisfying the requirements.
Any theorem in standard probability theory is a
necessary consequence of any environment in which the hypotheses of the
theorem itself are (a necessary consequence of the) requirements.

A particular model theoretical proof method consists in showing by 
model existence that a certain probability environment $(( \Omega, 
{\mathbb A},  \mathcal P), \mathcal R)$, exists; and then
proving that $(( \Omega, 
{\mathbb A},  \mathcal P), \mathcal R')$ is contradictory, where
$\mathcal R'$ equals $\mathcal R$ plus the negation of a statement
$s$. It follows that $s$ is a necessary consequence of the 
environment. When applied to Elementary Probability in Section \cite{Elementary}
below, this leads to a complete solution.

 \begin{ex} \label{UnifDist} The Uniform Distribution on $n$ points is a 
 Probability Environment
in which the requirements can be taken to be:
$\Omega=\{a_1, \dots, a_n\}$; 
$\mathbb A$ contains the $n+2$  symbols $\{\emptyset, \{a_1\}, \dots, \{a_n\},
\Omega\}$; $P$ is defined on a $\sigma$-algebra with all events being
 jointly perceivable; finally,  $P(\{a_k\})= c$, for each $\{a_k\} \in \mathbb A$  and some constant $c$.
To verify that this is indeed a Probability Environment it is enough to
take, for instance, $\overline \Omega:=\{1, \dots, n\},
\overline{\mathcal A}:= \mathcal P (\overline \Omega),$
$\overline P (\overline A) := |\overline A|/n$, and 
replace each $\{a_k\}$ by  $\{k\} $.

Alternatively: no requirements on $\Omega$;
$\mathbb A$ contains (at least) the $n+2$ symbols
$\emptyset, A_1, \dots, A_n, \Omega$;  $A_i \cap A_j = \emptyset$;
and $P(A_i)= P(A_j)$ for all $i \neq j$, $i,j=1,\dots,n$.
The concrete probability space above is again a model of the environment, 
with the replacement of $A_k$ by $\{k\} $
(this second formulation satisfies Tao's dogma about extendibility \cite{T}). 
\end{ex}

\subsection{Joint perceivability and mutual independence}

Specific requirements can be imposed for each different problem, but 
there are
standard ones, such as countable additivity or independence, 
which set probability theory apart from other theories. 
The imposition of such requirements is 
 facilitated by suggestive definitions.
This has always been the case with independence, which 
as mentioned plays the role of a requirement in applications,
and we now 
 introduce
a novel notion for additivity; among other things, 
it brings about the potential to unify diverse 
formulations of probability theory.

\begin{defn} \label{JointlyPerceivable}
In a probability environment $(( \Omega, 
{\mathbb A},  \mathcal P), \mathcal R)$,
a collection of events $\mathfrak {A}=\{ A_i\}_{ i \in \mathcal I}, 
A_i \in \mathbb A$, 
is {\bf jointly perceivable} (for  $P$) if $P$
is countably additive on 
$\sigma(\mathfrak {A})$, i.e
for every countable
subcollection of disjoint events $A_i \in \sigma(\mathfrak {A})$,
$ P(\cup_{ i =1}^{\infty} A_i)=
\sum_{  i =1}^{\infty}   P(  A_i)$. The events in $\mathfrak {A}$
are also called jointly perceivable.

\end{defn}
Notice that with this definition, additivity, now renamed joint perceivability, 
plays both the role of an assumption and that of a definition, as much
as independence is now doing. This parallelism is further developed here below.

When the requirements on $(\Omega, \mathcal A, P)$ are that $\mathcal A$ is a
$\sigma$-algebra of  jointly perceivable events, we say that
$(\Omega, \mathcal A, P)$ is a {\bf Kolmogorov probability space}.
That these requirements are consistent, and hence define a probability 
environment, can be verified by taking any $\overline \Omega$, 
$\overline {\mathcal A} =\{\overline \Omega, \emptyset\}$ and $\overline P=\delta_{\overline \Omega}$ 
(the semantic consistency check is essentially verabatim cited  from \cite{K}).

When the requirements on $(\Omega, \mathcal A, P)$ are that $\mathcal A$ is a
algebra and that all finite collections of events are jointly perceivable, we say that
$(\Omega, \mathcal A, P)$ is a {\bf finitely additive probability space}.

In other cases joint perceivability might hold for some but not for all finite 
 collections of events. This is the case in 
test spaces which appear in Quantum Mechanics 
\cite{FR72, W09, FL15}
as well as in other contexts, e.g. \cite{C10}.

The notion of joint perceivability has been phrased in a way that makes 
it comparable to the slightly adapted usual one of mutual independence.
\begin{defn} \label{Independence}
In a probability environment $(( \Omega, 
{\mathbb A},  \mathcal P), \mathcal R)$,
a collection of events $\mathfrak {A}=\{ A_i\}_{ i \in \mathcal I}, 
A_i \in \mathbb A$ 
is {\bf mutually independent} if for all disjoint classes $\mathfrak A_i$,
$i \in I$, $I$ any set of indices, $\mathfrak A_i \subseteq \mathfrak A$,
$P$ is countably moltiplicative on the product
$\otimes   _{i \in I} \sigma(\mathfrak A_i)$, i.e
for every countable
collection of events $A_i \in  \sigma(\mathfrak {A_i})$,
$ P(\cap_{ i =1}^{\infty} A_i)=
\prod_{  i =1}^{\infty}   P(  A_i).$ The events $A_i$ are called
mutually independent.
\end{defn}

To illustrate the parallelism between the concepts of joint perceivability and
mutual independence, we call finitely jointly perceivable a collection $\mathfrak A$
of
events in which additivity holds for all finite collections of elements of $\sigma(\mathfrak A)$,
 and
finitely mutually independent a collection $\mathfrak A$ for which 
factorization occurs for all finite products $\otimes_{i=1}^n \sigma(\mathfrak A_i)$
of disjoint collections $\mathfrak A_i \subset \mathfrak A$.

In some cases  both joint perceivability and mutual independence are finite:
take $\Omega = \mathbb N$ and $\mathfrak U$  an ultrafilter; then $P(A) =\delta_{\mathfrak U} (A)$
is both finitely mutually independent and finitely jointly perceivable on
$\mathfrak A = \mathcal P(\mathbb N)$, but neither is countable.
However, if one is countable and the other  finite, then the other is countable too.

\begin{thm} Let $(\Omega, \mathbb A, P, \mathcal R)$ be a probability
environment, and $\mathfrak A \subseteq  \mathbb A$.
If $\mathfrak A$ is jointly perceivable and finitely mutually independent, then
it is also (countably) mutually independent.

If $\mathfrak A$ is mutually independent and finitely jointly perceivable, then
it is also  (countably) jointly perceivable.
\end{thm}
\begin{proof}

(I) As $(\Omega, 
\sigma(\mathfrak A), P) $ is Kolmogorov,
the first statement follows from   standard probability theory (see, for instance, \cite{K02} pp. 51, 60).

\bigskip

(II) In the other direction, let $A^0=A^c$ and $A^1=A$, and consider 
$$
\mathcal F = \{\cap_{\ell=1}^k A_{\ell}^{\alpha_{\ell}} ,
k \in \mathbb N, A_j \in \mathbb Q, \alpha_j =0,1\} \cup \{\emptyset, \Omega\}.
$$
Clearly, $A, B \in \mathcal F $ implies $A\cap B \in  \mathcal F $;
$A=\cap_{\ell=1}^k A_{\ell}^{\overline \alpha_{\ell}}
\in \mathcal F$ implies $A^c= \cup_{\boldsymbol{ \alpha} \neq \boldsymbol{\overline \alpha} }
\cap_{\ell=1}^k A_{\ell}^{\alpha_{\ell}}$;
hence, $\mathcal F$ is a semialgebra containing $\mathfrak A$
and $P$ is finitely additive on $\mathcal F$.

To show that $P$ is countably additive on $\mathcal F$ consider $A \in \mathcal F$
such that $\cap_{\ell=1}^k A_{\ell}^{\alpha_{\ell}}
= A = \cup_{j=1}^{\infty} A(j), A(j) = \cap_{\ell=1}^{k_j} A_{\ell}^{\alpha_{j,\ell}}(j) \in \mathcal F$,
$A(j)$ disjoint. We now focus on the countable family
$$
\tilde {\mathfrak A} = \{A_i^{\alpha_{j,i}}(j), j \in \mathbb N, 1 \leq i \leq k_j \}
\subseteq \mathfrak A.
$$
Let's fix an order of the elements of $
\tilde {\mathfrak A} $ and relabel them $B_1, B_2, \dots$. Then we consider the 
map $T:\Omega \rightarrow \{0,1\}^{\mathbb N}$, such that 
$T(\omega) =( \delta_{B_1}(\omega),  \delta_{B_2}(\omega), \dots)$.
$T$ is measurable with respect to the Borel $\sigma$-algebra
in $ \{0,1\}^{\mathbb N}$, and the $\sigma$-algebra $\sigma(\tilde {\mathfrak A}) $.
In fact, for each cylinder $C=C_{i_1, \dots, i_k}^{\beta_1, \dots, \beta_k}
=\{\rho \in \{0,1\}^{\mathbb N} : \rho_{i_m}=\beta_{i_m} \}$ we have
$T^{-1}(C) = \cap_{m=1}^k B_{i_m}^{\beta_m}$. It follows that $\sigma=
T(P)$   is a finitely additive
probability on $\{0,1\}^{\mathbb N}$; furthermore,
 $\sigma(C_{i_1, \dots, i_k}^{\beta_1, \dots, \beta_k})
=P(\cap_{m=1}^k B_{i_m}^{\beta_m}) = \prod_{m=1}^k P(B_{i_m}^{\beta_m}) $
by independence of the $B_j$'s under $P$.
Hence, for $H_k = \{0,1\}, \mathcal A_k = \mathcal P(H_k), \gamma_k$ the countably
additive probability on $\mathcal A_k$ such that $\gamma_k(1)=P(B_k)$,
$H=H^{\infty}$, $\sigma$ is a finitely additive probability on $H$ such that 
for $D_k \subseteq H_k$
\begin{eqnarray} \label{sigma}
\sigma( \times_{\ell =1}^{\infty} D_{\ell}) &=& P(T^{-1}( \times_{\ell =1}^{\infty} D_{\ell}))
\nonumber \\
&=& P(\cap_{\ell =1}^{\infty} T^{-1} (D_{\ell})) = \prod_{k=1}^{\infty} P(T^{-1}(D_k))
=  \prod_{k=1}^{\infty}  \gamma_k(D_k)
\end{eqnarray}
again by countable independence of $P$.

These are the conditions used in \cite{D74, PS76}, see also \cite{K82}, to show that 
there exists a unique finitely additive probability $\tilde P$, satisfying the further condition
 \eqref{conditionPurves}
below, such that \eqref{sigma} holds for $\tilde P$. As \eqref{sigma} holds for $\sigma$,
if it satisfies the condition below then $\sigma = \tilde P= \otimes_{k=1}^{\infty} \gamma_k$.
Hence, $\sigma$ is countably additive on the Borel $\sigma$-algebra of $H$.
It the follows that $P$ is countably additive on $\sigma(\tilde {\mathfrak A})$.
The condition to check from \cite{D74, PS76} is that for all clopen subsets $D$ of 
$\times_{\ell=k+1}^{\infty} H_k$,
\begin{eqnarray} \label{conditionPurves}
\sigma(D) = \int_{\times_{\ell=1}^{k} H_k} \sigma(D(x_1, \dots, x_k)) 
d\otimes_{\ell=1}^{k} \gamma_k((x_1, \dots, x_k))
\end{eqnarray}
where $D \subseteq H, D(x_1, \dots, x_k)
=\{z=(z_1, z_2, \dots) \in H: (x_1, \dots, x_k, z_1, z_2, \dots) \in D\}.$ In the 
present case,  \eqref{conditionPurves} holds by independence and finite additivity, as
\begin{eqnarray} \label{sigma}
\sigma( D) &=& \sigma( \cup_{(x_1, \dots, x_k)} (D(x_1, \dots, x_k) 
\cap (\rho_1=x_1, \dots, \rho_k=x_k)      ) \nonumber \\
&=&\sum_{(x_1, \dots, x_k)}  \sigma( (D(x_1, \dots, x_k) 
\cap (\rho_1=x_1, \dots, \rho_k=x_k)      )    \\
&=& \sum_{(x_1, \dots, x_k)}  \sigma( (D(x_1, \dots, x_k) )
\gamma_k( (\rho_1=x_1, \dots, \rho_k=x_k)      ) \nonumber \\
&=&  \int_{\times_{\ell=1}^{k} H_k} \sigma(D(x_1, \dots, x_k)) 
d\otimes_{\ell=1}^{k} \gamma_k((x_1, \dots, x_k))
\nonumber
\end{eqnarray}
If $ {\mathfrak A}$ is countable then the proof would be finished. For general
$\mathfrak A$ we observe that the countable additivity of $P$ on 
$\sigma(\tilde {\mathfrak A})$ for each $\tilde {\mathfrak A}$ implies that $P$ is 
countably additive on $\mathcal F$.

\bigskip

(III) Consider now a Stone representation  \cite{YH52, S16} 
in which   for a 
finitely additive probability $\mu$
 on a measurable space $(\Omega, \mathcal A)$, with
 $\mathcal A$ a $\sigma$-algebra,
 there are a compact measurable space $(\hat \Omega, \hat{\mathcal A})$,
 and a measurable map $\psi: \Omega \rightarrow \hat \Omega$, with 
 $\psi(\Omega) $ dense in $ \hat \Omega$,
 such that $\psi(E) $ has a unique extension $\hat E \in \hat{\mathcal A}$,
 and there is a unique countably additive probability $\hat \mu$
 on $(\hat \Omega, \hat{\mathcal A})$ determined by $\hat \mu (\hat E) = \psi(\mu) (\psi (E))
 $ for each $E \in \mathcal A$. Notice that the extension is monotone as
 $\hat E$ can be defined as the closure, in a suitable topology, of $\psi(E)$:
$E_1 \subseteq E_2$ implies that $\hat{E_1} \subseteq \hat{E_2}$
as 
(see \cite{S16} \cite{YH52}). 
 The probability
 $\hat{\mu} (\hat \Omega \setminus \psi(\Omega))$ of the corona $\hat \Omega \setminus \psi(\Omega)$ 
is the deficiency of $\mu$  \cite{S16}.

\bigskip

(IV) As $P$ is   countably additive on a semialgebra $\mathcal F$ generating 
$\sigma(\mathfrak A)$ from
Part (II) above,
then it has a unique countably additive extension $P^{ca}$ to $\sigma(\mathfrak A)$
(by standard extension theorem \cite{K02}). 

From Part (III) we have 
$\hat P$  defined, and countably additive, on $\hat{\sigma(\mathfrak A)}$.
Let $\mathcal L=\{E \subseteq \Omega: P^{ca}(E) = \hat P(\hat E)\}$. Clearly,
$\mathcal F \subseteq \mathcal L$, as for each $E \in \sigma(\mathfrak A)$
$P^{ca}(E) =P(E)= \hat P(\hat E)$, and 
$\mathcal F$ is a $\pi$-system. Moreover, $A, B \in  \mathcal L, A \subseteq B$
implies 
\begin{eqnarray}
P^{ca}(B\setminus A) &=& P^{ca}(B) - P^{ca}(A) \nonumber \\
&=&\hat P(\hat B) - \hat P(\hat A)=\hat P(\hat B \setminus \hat A)
\nonumber
\end{eqnarray}
where the last equality holds as $\hat A \subseteq \hat B$ by the monotonicity of the 
$\hat{}$ extension; also for an increasing sequence
$A_i \in  \mathcal L$ 
\begin{eqnarray}
P^{ca}(\cup_{i=1}^{\infty} A_i) &=&\lim_i P^{ca}( A_i) \nonumber \\
&=&\lim_i \hat P( \hat A_i) =\hat P(\cup_{i=1}^{\infty} \hat A_i) 
\nonumber
\end{eqnarray}
where again the last equality holds  by the monotonicity of the 
$\hat{}$ extension.
Hence, $\mathcal L$ is a $\lambda$-system, and  the $\pi$-$\lambda$-theorem
implies that $\sigma(\mathfrak A) \subseteq \mathcal L$. It follows that for all 
$E \in \sigma(\mathfrak A)$, $P^{ca}(E)= \hat P(\hat E) = P(E)$, i.e. $P^{ca}=P$,
and $P$ is countably additive on $\sigma(\mathfrak A) $.
\end{proof}
This theorem underlines once again the fact that results about independent
sequences which are valid in 
a countably additive setting
can be proven in the finitely additive setting as well (see, e.g. \cite{K82}).

\subsection{Arbitrage or Dutch Books}
If a contradiction is derived, by deductive rules, in a
probability pre-environment, then this is inconsistent.
This derivation can be eased on some occasions by the
  method  of Arbitrages, or Dutch Books.
 Informally, a Dutch Book is  a rigging strategy 
in which an individual is lead to believe that a certain game is worth playing, while (s)he is losing
some strictly positive amount every time;
equivalently, it can be defined as a betting scheme to extract a sure profit
from an incoherent agent forced to accept any bet on his betting quotients
\cite{DB}.
 More
formally,
\begin{defn}
 Given a probability pre-environment, a {\bf weak Dutch Book} 
 against the believer of the pre-environment is a 
 an additional random variable $V$, with expectation operator $E$, added to 
 the probability pre-environment, with  the additional requirements that
 \begin{enumerate}
\item if $X=\mathbb I_A$, the indicator function of an event in $\mathbb A$,
then 
$E(\mathbb I_A) = P(A)$;
\item $E$ is linear on the indicator functions;
\item if $X \geq Y$ are random variables on which $E$ is defined,
then $E(X) \geq E(Y)$;
\item $V \leq 0$;
\item $E(V)>0$.
\end{enumerate}
 A {\bf (strict) Dutch Book} is as above,
 but with 4. and 5. replaced by
 $4'$. $V \leq -1$
 and $5'$. $E(V)\geq 0$, respectively.
 In case a Dutch Book exists we call a believer of the inadmissible requirements  an {\bf incorrect evaluator} of probabilities
 \end{defn}

\begin{ex}
If we require that an event $A$ has $P(A) +P(A^c)=2$ and $A$ and $A^c$ are
jointly perceivable,
then let $V=\mathbb I_A + \mathbb I_{A^c}-2$. For any basic probability space
$(\overline \Omega, 
\overline{\mathbb A}, \overline P)$
and any $\omega \in \overline \Omega$, $V(\omega)=-1$, but
based on the  requirements of the pre-environment $E(V)=P(A) +P(A^c)-2=0$. 
So $V$ is the a strict Dutch Book.

\end{ex}
In a limited form, use of Dutch Books to define probability has been proposed
by De Finetti \cite{DF1, DF2}.

If $V$ is a strict Dutch Book then $V-1$ is a weak  Dutch Book. Moreover,
if there is a weak Dutch Book then no basic probability space 
satisfying the requirements of a pre-environment can exist,
as for any random variable $V$ on  a basic probability space
with $V\leq 0$ it holds that, 
whatever the definition of expectation, $E(V) \leq0$ by monotonicity of expected values.

In some cases, such as for finitely many requirements on finitely many events,
 also the opposite holds, and absence of a Dutch Book
guarantees the existence of the environment, see Section \ref{PSDB} below.
In general, the situation is more complex: in Section \ref{MP}
below we see that for countably many requirements
 absence of Dutch Books can be compatible with distributions on
 hyperreals, while no standard distribution exists.

\bigskip

\section{Elementary Probability} \label{Elementary probability}
In this section we consider the theory of probabilities for finitely many events
from the point of view of starting from a collection of assumptions (i.e.
requirements for a probability environment) and looking for a model satisfying them
(i.e. checking semantic consistency).
After observing that most problems can be expressed in terms of real variables,
we define Elementary Probability as the collection of problems involving
finitely may algebraic relations, and show its decidability.

\subsection{Probabilities involving a finite number of events}
A general framework for dealing with finitely many events consists of taking
the following requirements for a probability pre-environment 
 $(\Omega, \mathbb A, P)$ (i.e. fixing some  symbols and imposing requirements on them):  
\begin{enumerate} 
\item
no requirements on $\Omega$; 
\item  
$\mathbb A$  contains at least $n+2$ events 
$\Omega,
A_1, \dots, A_n, \emptyset$;
\item all finite collections of events are jointly perceivable under $P$, 
i.e. $P$ is fully additive;
\item further requirements on $P$ are determined by a  collection of expressions of the form
\begin{eqnarray} \label{RawEquations}
g_r=g_r(P(B_1(A_1, \dots, A_n)), \dots, P(B_{k(r)}(A_1, \dots, A_n))) \triangleleft 0,
r \in  R,
\end{eqnarray}
where $R$ is a set of indices of any possible cardinality,
 the $g_r$'s are real valued functions,
the $B_{j}(A_1, \dots, A_n)$'s, $j=1, \dots,k(r)$, are  
boolean combinations
 of some of the 
$A_1, \dots, A_n$'s, 
$k(r)$ is an integer,
and $ \triangleleft$ indicates one of $=, \neq,  \geq$
(notice that all other inequalities, including $>$, can be obtained combining
relations with the above values of $ \triangleleft$).
 \end{enumerate}

\begin{lem} \label{ChangeOfVariables}
The above family of requirements is semantically consistent, i.e.
determines a probability environment, if and only if the following happens.

 For every $j=1, \dots, k$ 
  let $B_j$ be expressed in
 disjunctive normal form $B_j=\cup_{\alpha \in \Sigma_j} A^{\alpha}$ \cite{HM}
 for the appropriate $\Sigma_j \subseteq \Sigma= \{-1,1\}^n$,
 $A^{-1}=A^c$, $A^1=A$ and $A^{\alpha}= \cap_{i=1}^n A_i^{\alpha_i}$.
Consider then the change of variables  $x_j= \sum_{\alpha \in \Sigma_j} y_{\alpha}$, 
 using the $2^n$ variables ${\bf y}=\{y_{\alpha}\}_{\alpha \in \Sigma}$, one for each of the 
 $A^{\alpha}$. 
 Then the family of requirements is admissibile if and only if
 the system of  equations and inequalities
 \begin{eqnarray} \label{system}
 \begin{cases}
 g_r(x_1({\bf y}), \dots, x_k({\bf y}))=g_r(\sum_{\alpha \in \Sigma_1} y_{\alpha}, \dots, \sum_{\alpha \in \Sigma_k} y_{\alpha})
\triangleleft 0  \\
\sum_{\alpha \in \{-1,1\}^n} y_{\alpha}=1\\
y_{\alpha}
 \geq 0
 \end{cases}
 \end{eqnarray}
obtained by
 the change of variables $x_j = x_j({\bf y})$, together with
 the additional  conditions of normalization and nonnegativity,
 admits  a (real) solution $y=(y_{\alpha})_{\alpha \in \Sigma}$.
 \end{lem}
 
 \begin{proof} Clearly, if there is a concrete probability space $
 (\overline \Omega, \overline {\mathcal A}, \overline P)$ with events
 $\overline A_i \in \overline {\mathcal A}$ in one to one correspondence with 
 the $A_i$'s, satisfying all the requirements, the values $y_{\alpha}
 = P(A^{\alpha})$ form a set of solutions to the the system 
 $g_r(x_1({\bf y}), \dots, x_k({\bf y}))
\triangleleft 0 $, as all events are jointly perceivable (i.e. 
$P$ is fully additive).

Viceversa, if a solution ${\bf y}=\{y_{\alpha}\}_{\alpha \in \Sigma}$
 exists, then take $\overline \Omega=
\{-1,1\}^n$, for each $\overline \omega \in \overline \Omega$ let $\overline P(\overline \omega)=
y_{\overline \omega}$,  $\overline{A_i}=\{\overline \omega: \overline \omega_i = 1\}$,
 $\overline{ \mathcal A} $ equal to the $\sigma$-algebra generated by the
collection of the $\overline{A_i}$'s, and, finally, $P$ additive (which implies that
 all events jointly perceivable.
It is easy to verify that $P(\emptyset)=0, P(\Omega)=1$, $P$ is monotone, and
satisfies all the requirements (i.e.
there is a basic probability space realizing the environment).
The requirements are then admissible are requested.

 \end{proof}
 
 Notice that a solution of \eqref{RawEquations} expressed in terms of the
 ${\bf x}$ variables
 does not imply consistency of the requirements, as these
 relations are still missing the requirements about additivity (i.e.  joint perceivability)
 and non negativity of probabilities; only absence of 
 a solution could be used to ascertain inconsistency.
 
Lemma \ref{ChangeOfVariables} suggests a classification of
probability environments for finitely many events in terms 
 of  the number of equations and inequalities and the type of functions
appearing in them. 
Some problems, such as maximal entropy, involve  uncountably many or non polynomial $g_j$'s;
in most situations, however,  
the requirements involve
 only finitely many polynomial equations and inequalities.
 In addition, in all problems involving macroscopic events
 $P$ is naturally taken as additive (i.e. all events are jointly perceivable).
 It is
 natural to call
 this class of problems {\bf (Classical) Elementary Probability}.

\subsection{Algebraization 
and decidability of Elementary Probability} \label{AlgebAndDec}

 In Elementary Probabilty, Requirement $4.$ above
 becomes
 \begin{enumerate} 
\item[$4'$.]  the requirements on $P$ are determined by a 
  finite collection of expressions of the form
\begin{eqnarray}\label{ReqElemProb}
\sum_{0 \leq \rho_1,\dots, \rho_k \leq s }
a_{\rho_1, \dots, \rho_k} (r) \prod_{j=1}^k (P(B_{j}(A_1, \dots, A_n) ) )^{\rho_j}
\triangleleft 0, \quad r \in R
\end{eqnarray}
where the $B_{j}(A_1, \dots, A_n)$'s are    boolean combinations of some of the 
$A_1, \dots, A_n$'s, $\triangleleft \in \{\geq, =, \neq \} $,
 $a_{\rho_1, \dots, \rho_k} (r) \in \mathbb R$, 
  $R \subset \mathbb N$   is a
 finite set of integers, the $\rho_j$'s, $s$ and $k$ are  integers. 
 \end{enumerate}

\begin{cor} \label{equiv2}
The consistency problem for probability environments in
Elementary Probability is decidable.
\end{cor}
\begin{proof}
If there are only a finite number of equations
  and inequalities involving polynomial $g_r$'s, then
 Lemma  \ref{ChangeOfVariables}
  implies that
the admissibility problem  is equivalent to 
the nonemptyness of the semialgebraic set defined by the polynomial relations
$g_r=g_r(({\bf x}({\bf y}))) \triangleleft 0$,
$r=1, \dots, m$, together with 
$ \sum_{\alpha \in \{-1,1\}^n} y_{\alpha}=1$ and $y_{\alpha}
 \geq 0$,
in the variables $y_{\alpha}$'s.

Using  Tarski-Seidenberg elimination  and Sturm's theorem
\cite{BCR},
the existence of a solution is decidable in a finite number of steps.

\end{proof}

Notice that  Tarski-Seidenberg  and
Sturm's theorems are purely existential results,
establishing existence
or absence
of solutions of polynomial equations  and inequalities 
 even in cases in which
the solutions cannot be explicitly found. 

We now have a procedure to check consistency  in Elementary Probability:
state the assumptions of a problem or a potential application, 
write them in algebraic form, check consistency by Tarski-Seidenberg elimination
or an alternative algorithm \cite{BPR06}, proceed with derivation of 
(now safely consistent) consequences as usual.
As a very simple example, in Appendix \ref{analysis} the contradictory problem of Appendix \ref{Contradictory} is
formally analyzed by means of algebraization. Another example is
in Appendix \ref{Example}.

\bigskip

We can, however, make a further step
as probability environments allow model theoretical  proofs
of necessary consequences of the requirements
(i.e. derivation of a consequence if all models satisfy it).
Indicating the negation of a relation $g(\mathbf x) \triangleleft 0$, with
$\mathbf x \in \mathbb R^d$, by 
 $g(\mathbf x) \bcancel{ \triangleleft } \, 0$,
 we have:
\begin{cor} \label{equiv3}
If 
\begin{eqnarray} \label{newrelaz}
g_{|R|+1}(P(B_1(A_1, \dots, A_n)), \dots, P(B_k(A_1, \dots, A_n))) \triangleleft 0
\end{eqnarray}
is an elementary probability relation, then it is a necessary consequence of
the probability environment described by 
the requirements $1.,2., 3.,4.'$ if and only if the system   \eqref{system}
augmented by the relation 
 $g_{|R|+1}(x_1(\mathbf y), \dots, x_k(\mathbf y)  )\bcancel{ \triangleleft } \,  0$
 admits no solutions. Such consequentiality  is then decidable.
\end{cor}
\begin{proof}  
In a probability environment the system \eqref{system} has at least one solution.  If one such
solution is also a solution of the augmented system, then there exists a basic
probability space in which the requirements and the negation of  \eqref{newrelaz} hold, 
hence the statement cannot be a 
necessary consequence. Viceversa, if no solution of \eqref{system} solves
the augmented system, then \eqref{newrelaz} holds in all the basic probability spaces 
which are probabilistic models of the requirements, hence it is a necessary consequence.

As the existence of solutions is decidable, so is the above deduction rule.
\end{proof}  

This allows to change the last step in the solution of problems in Elementary Probability:
express the negation of the potential consequence in algebraic form, 
use again Tarski Seidenberg elimination or another algorithm to verify that 
there is no longer a solution. If it is so, then the consequence is proven.
See the last part of Appendix \ref{Example} for an example.

\bigskip

Albeit NP complete, the method in
Corollary \ref{equiv3} solves thus
 "all problems" in Elementary Probability, at least in principle.   The same claim
 has been made by Boole \cite{B54},  without being able to complete his program.

\bigskip

\subsection{Relation with semialgebraic geometry}
Semialgebraic sets of any degree emerge in discussing satisfiability
in Elementary Probability, for instance with mutual independence of 
many events.
On the other hand, it is easy to see that 
  each semialgebraic set 
 included in some nonnegative $n$-dimensional simplex
 of the form $\Sigma_k=
 \{(x_1, \dots, x_k): x_i \geq 0, \sum_{i=1}^k x_i = 1\}$
 can be interpreted
 as a description of admissibility of requirements for some probability
 environment with jointly perceivable events (i.e. fully additive probability). 
 It is possible to use, for instance, disjoint events.
  
 \begin{thm}\label{DB1}
Each semialgebraic set  included in some nonnegative simplex
 of the form $\Sigma_k$ can be expressed as the set of conditions
for satisfiability of a probability environment in classical
Elementary Probability.

\end{thm}
\begin{proof}
Let $ g_{r}(x_1, \dots, x_k) \triangleleft 0, r=1, \dots, m+k+1$, be a system of 
polynomial relations describing a semialgebraic set
included in $\Sigma_k$. We can always assume that the last relations are
$x_j \geq 0$ for $j=1, \dots, k$, and $\sum_{j=1}^k x_j = 1$; in addition,
we have
 \begin{eqnarray}\label{g's}
  g_{r}(x_1, \dots, x_k) = \sum_{\rho_1, \dots, \rho_k: \, 0 \leq \rho_i \leq s\text{ for } i=1, \dots, k}
a_{\rho_1, \dots, \rho_k} (r) \prod_{j=1}^k x_j^{\rho_j}
\end{eqnarray}
for $r=1, \dots, m$, where $s$ is the overall maximal degree of any variable in
any of the polynomials $g_{r}, r=1, \dots, m$, and $a_{\rho_1, \dots, \rho_k} (r)$ are,
possibly zero, coefficients. 

Next, for the given $k$, consider the  requirements $1., 2., 3.$ and $4.'$ for a probability environment
in which $n:= k $, $B_j=B_j(A_1, \dots, A_k):=A_j\cap \cap_{i \neq j}A_i^c$, and the
polynomial relations in $4.'$ are $g_{r}(P(B_1), \dots, P(B_k)) \triangleleft 0$.

As in  Lemma \ref{ChangeOfVariables}, consider the variables $x_j:=P(B_j), j=1, \dots, k$
and the variables $y_{\alpha}$. For $\alpha(j)= (2\delta_{1=j}-1, \dots, 2\delta_{k=j}-1)$
we have $x_j=y_{\alpha(j)}$. The complete
system of polynomial relations becomes  $g_{r}(y_{\alpha_1}, \dots, y_{\alpha_k}) \triangleleft 0, 
r=1, \dots, m$, $\sum_{j=1}^k y_{\alpha_j}=1$,
$\sum_{\alpha} y_{\alpha}=1$, and $y_{\alpha}\geq 0$ for all $\alpha$. Combining the last three sets of 
relations, one gets that necessarily $y_{\alpha}=0$ for all $\alpha \neq \alpha(j)$ for all $j$. 
Hence, only the relations $g_{r}(y_{\alpha(1)}, \dots, y_{\alpha(k)}) \triangleleft 0, 
r=1, \dots, m$, $\sum_{j=1}^k y_{\alpha(j)}=1$ and $y_{\alpha(j)}\geq 0$ for
$j=1, \dots, k$ are left,
which  form a system coinciding with the original one.
 
 \end{proof}

 The requirements formulated to reproduce a general semialgebraic set
have no real probabilistic content, but depending on the specific case, 
one can sometimes obtain more meaningful  problems.

\section{ Elementary Probabilities via Dutch Books }
\label{PSDB}

In this section we prove that if a finite number of polynomial requirements 
are stated upon probabilities of boolean
combinations of finitely many events, then the requirements determine 
a probability environment
if and only if no Dutch Book can be realized against the believer
of such requirements. Besides its intrinsic interest, 
a reason for developing such equivalence is that it may be 
computationally advantageous in certain cases
with respect to Tarski Seidenberg elimination or related algorithms
\cite{Pa, BPT, L10}. For the definitions see Section \ref{Probability}.

\subsection{Dutch Books in Elementary Probability }
 
\begin{thm}\label{DB1}
The requirements of an elementary probability pre-environment 
with $n$ events are not 
 consistent if 
and only if, assuming that it is possible to realize
a, finite but sufficiently large, number of i.i.d.,
joinly perceivable, copies of 
the collection of events,
it is possible to realize a weak Dutch Book against any incorrect evaluator
believing such requirements.

\end{thm}
Some care must be used in interpreting the content of this theorem. 
When talking about an (incorrect) evaluator of elementary probability
we intend that (s)he has determined a phenomenon in which (s)he
can identify the various events which enter into the requirements.
One of the assumptions in the theorem is that it is possible to 
find or produce a, finite but sufficiently large, number of
phenomena in each of which the evaluator is lead to identify 
"copies" of the original events, in such a way that the original and all these
copies are jointly perceivable and collectively independent.
\begin{proof}
Consider requirements of the type $1., 2., 3.$ and $4'.$ for a problem in 
Classical Elementary
Probability, involving events $A_{i_1}, i_1=1, \dots, n$. Feasibility of
the requirements is equivalent, by Corollary \ref{equiv2},
to nonemptiness of the semialgebraic set defined by 
the polynomial relations $g_r=g_r(({\bf x}({\bf y}))) \triangleleft 0$,
$r=1, \dots, m$, together with 
$ \sum_{\alpha \in \{-1,1\}^n} y_{\alpha}=1$ and $y_{\alpha}
 \geq 0$,
in the variables $y_{\alpha}$'s.

By distinguishing the three possible values of $\triangleleft$, 
we can assume that the polynomial relations can be expressed as follows:
\begin{eqnarray} \label{system1}
\begin{cases} 
f_r({\bf y}) =0, \quad r=1, \dots, m_1 \\
g_r({\bf y}) \geq 0, \quad r=m_1+1, \dots, m_1+m_2\\
h_r ({\bf y})\neq 0, \quad r=m_1+m_2+1, \dots, m_1+m_2+m_3=m'.
\end{cases}
\end{eqnarray}
By the positivstellensatz (see \cite{Kr, S, BCR}), the system has no solution if 
and only if the following happens. There exists a polynomial $F$
in the ideal generated by the $f_r$'s in
$  \mathbb R[{\bf y}]$, a polynomial $G$
in the cone generated by the $g_r$'s in
$  \mathbb R[{\bf y}]$ and a polynomial $H$
in the multiplicative monoid generated by the $h_r$'s in
$  \mathbb R[{\bf y}]$
such that 
\begin {eqnarray}\label{polynomialdutchbook1}
F+G+H=0.
\end{eqnarray}
More explicitely,
there are polynomials $t_r \in \mathbb R[{\bf y}],  r =1, \dots, m_1 $;
$s_J \in \mathbb R[{\bf y}], J \subseteq \{m_1+1, \dots, m_1+m_2\}$
which are sums of squares;
and even  integers 
$
k_r, r=m_1+m_2+1, \dots, m_1+m_2+m_3$, such that 
\begin{eqnarray}\label{polynomialdutchbook2}
v({\bf y})=
\sum_{r=1}^{m_1} t_r f_r +
\sum_{J \subseteq \{m_1+1, \dots, m_1+m_2\}} s_J \prod_{r \in J} g_r+
\prod_{r=m_1+m_2+1}^{m_1+m_2+m_3} (h_r )^{k_r} =0
\end{eqnarray}  (\cite{BCR}).

We need to investigate the polynomial (\ref{polynomialdutchbook2})
as a polynomial in the $y_{\alpha}$'s before
taking into account that all its coefficients are zero. As such,
let $\nu_{\alpha}$ be maximal power of the variable 
$y_{\alpha}$, and consider $\nu= \sum_{\alpha  \in
\{-1,1\}^{n} }
\nu_{\alpha}$; next, list the $\alpha$'s in some fixed order
$\alpha_1, \dots, \alpha_{\bar n}$; for each $\gamma \in \{1, \dots, 
\bar n \}$ 
let $\Sigma^{(\alpha_{\gamma})}$ be the set of all permutations
$\sigma^{(\alpha_{\gamma})}=(\sigma^{(\alpha_{\gamma})}_i), 
i=1, \dots, \nu_{\alpha_{\gamma}}$
of integers 
\begin{eqnarray}\label{integers}
\{\sum_{{\gamma}'=1}^{{\gamma}-1} \nu_{\alpha_{{\gamma}'}}+1,
\dots, \sum_{{\gamma}'=1}^{{\gamma}} \nu_{\alpha_{{\gamma}'}} \}. 
\end{eqnarray} 

We take $\nu$ independent, jointly perceivable copies of the events 
identified by the incorrect evaluator. We then form  a random variable,
basically by replacing each occurrence of the variables $y_{\alpha}$'s  in (\ref{polynomialdutchbook2})
by the indicator function $\mathbb I_{\alpha,(j)}$
 that the $j$-th independent copy, 
 with $j$ to be determined, of the event $A^{\alpha} $
takes place, and then summing 
the fully replaced polynomial over all permutations
of the indices of the copies. We need to specify how to choose the
copy to be used for each replacement.
We do this in steps for each selection 
$\{\sigma^{(\alpha ) }\}_{\alpha  \in\{-1,1\}^{n} }$ of a permutation for each
$\alpha  $:
\begin{enumerate}
\item consider each of the polynomials $ t_r f_r ,
s_J \prod_{r \in J} g_r$ and $\prod_{r=m_1+m_2+1}^{m_1+m_2+m_3} (h_r )^{k_r}$
separately. 
\item In each such polynomial $u$  consider one of its factors at
a time using the factorization in which they are already expressed
(for instance $t_r$ and $f_r$ for those in the ideal);
\item 
expand out completely each such factor into a sum of monomials;
consider each monomial separately and  if in  such a monomial
the variable 
 $y_{\alpha}$ appears at some power $\overline m^{(1)}_{\alpha}$,
 then replace it by the product 
 $\prod_{j=1}^{\overline m^{(1)}_{\alpha} } \mathbb I_{\alpha,(\sigma^{(\alpha  )}_j)}$;
 repeat for all variables in $\bf y$. 
 Decorate the symbol of the factor by a tilde 
  to indicate the  random
 variable thus obtained, so that $t_r$ is changed into $\tilde t_r=\tilde t_r^{\sigma^{(\alpha  )}}$, 
 for instance; 
 notice that, although we will drop the dependency, the random variable depends on the fixed permutation $\sigma^{(\alpha  )}$.
 \item Consider the second factor of $u$; repeat the previous step,
 with this change: if the variable 
 $y_{\alpha}$ appears at some power $\overline m^{(2)}_{\alpha}$,
 then replace it by the product 
 $\prod_{j=\overline m^{(1)}_{\alpha}+1}^{\overline m^{(2)}_{\alpha} } \mathbb I_{\alpha,(\sigma^{(\alpha  )}_j)}$.
 \item Repeat, always using 
 $\mathbb I_{\alpha,(\sigma^{(\alpha  )}_j)}$ referred to 
 new $j$'s and hence additional copies,
 till all factors of $u$ have been changed; notice that the
 total number of copies of $A^{\alpha} $
 used in the procedure
 is not greater than $\nu_{\alpha}$.
 \item Consider the next polynomial from the list in point $1.$, and repeat
 steps $2.$-$5.$ till all polynomials in $1.$ have been changed.
 \end{enumerate}
 The above procedure produces a random variable 
\begin{eqnarray*}
 \tilde u(\sigma^{(\alpha  )})= \sum_{r=1}^{m_1} \tilde t_r \tilde f_r +
\sum_{J \subseteq \{m_1+1, \dots, m_1+m_2\}} \tilde s_J \prod_{r \in J} \tilde g_r+
\prod_{r=m_1+m_2+1}^{m_1+m_2+m_3} (\tilde h_r )^{k_r}.
\end{eqnarray*}

 Let then 
$V = \sum_{\text{ all permutations}\{\sigma^{(\alpha  )}\}_{\alpha  \in
\{-1,1\}^{n} }}
\tilde u (\sigma^{(\alpha  )})$.

We compute the expected value of $\tilde v$ according to the incorrect evaluator.
Consider one of the polynomials in point $1.$ after substituting
the variables with the indicator functions as above; 
any  two of its factors contain indicator functions which refer to different
copies of the space, by \eqref{integers}.
Therefore, the incorrect evaluator would consider all factors as
independent, and factorize the expected value of the product.
Similarly, in each monomial inside each factor, the variables
were also substituted with indicator functions which refer to different
copies of the space by $3.$ and the fact that the $\sigma$'s are permutations;
 hence, indicator functions in each monomial  are considered independent by the incorrect evaluator.
We have thus that he/she would compute the expected value of the product of
the indicator functions which has replaced the variables of a monomial 
$\prod_{\alpha} y_{\alpha}^{k_{\alpha}}$ as
$$
E(\prod_{\alpha} \prod_{j \in J_{\alpha} }
\mathbb I_{\alpha,\sigma^{(\alpha)}_j})
=\prod_{\alpha}  (P(A^{\alpha}))^{ k_{\alpha} }
$$
for  some  set of distinct integers $J_{\alpha}$ of cardinality
$k_{\alpha}$. In addition, all events are jointly perceivable,
so that the expectation is linear on the sum of monomials.
Hence, 
 the incorrect evaluator would compute $E(\tilde u(\sigma^{(\alpha  )}))$ as the corresponding
 polynomial in ${\bf y}$ with the $y_{\alpha}$'s replaced by
 the $P(A^{\alpha} )$'s. This is the value for
 which he/she thinks that the relations in (\ref{system1}) hold.
It follows that, if $\overline {\bf y}$ indicates the value of 
 ${\bf y}$ with the above substitutions,
 the incorrect evaluator would compute, again by linearity of the expected value
 due to joint perceivability:
\begin{eqnarray*}
E(V)&=&
\sum_{\text{ all permutations }\{\sigma^{(\alpha  )}\}_{\alpha  \in
\{0,1\}^{n} }}
E(\tilde u(\sigma^{(\alpha  )})) \\
&=&\sum_{\text{ all permutations}}   \left(
\sum_{r=1}^{m_1} E(\tilde t_r)E( \tilde f_r )+
\sum_{J \subseteq \{m_1+1, \dots, m_1+m_2\}} E(\tilde s_J )\prod_{r \in J} E(\tilde g_r)
\right. \\
&&\quad \quad \quad \left. +
 \prod_{r=m_1+m_2+1}^{m_1+m_2+m_3} (E(\tilde h_r ))^{k_r} \right) \\
&=&  \left(  \prod_{\alpha  } \nu_{(\alpha  )} ! \right) \left(
\sum_{r=1}^{m_1}  t_r(\overline {\bf y}) f_r (\overline {\bf y})+
\sum_{J \subseteq \{m_1+1, \dots, m_1+m_2\}} s_J (\overline {\bf y})
\prod_{r \in J}  g_r(\overline {\bf y})
\right. \\
&&\quad \quad \quad \left. + 
\prod_{r=m_1+m_2+1}^{m_1+m_2+m_3} ( h_r (\overline {\bf y}))^{k_r} \right) 
\\
&\geq& \left(  \prod_{\alpha  } \nu_{(\alpha  )} ! \right) 
\prod_{r=m_1+m_2+1}^{m_1+m_2+m_3} ( h_r (\overline {\bf y}))^{k_r} 
>0
\end{eqnarray*}
from the equalities and inequalities in (\ref{system1}), and the 
properties of the polynomials $s_J$ and the powers $k_r$.

\bigskip

We finally evaluate $V$ for each possible realization of events
 in the collection 
identified by the incorrect evaluator and in all the copies. 
First, expand $\tilde v$ completely, and then collect
all terms corresponding to random variables which have
replaced the same monomial
$\prod_{\alpha} y_{\alpha}^{k_{\alpha}}$.
For each such monomial, there is a certain number $m$ of terms,
with coefficients $c_1, \dots, c_m$; 
we have $\sum_{i=1}^m c_i=0$ as the corresponding
monomial in  the expansion of the l.h.s of  (\ref{polynomialdutchbook2}) has 
zero coefficient. When taken with the indicator functions replacing the
variables, the sum is not immediately zero, as the
indicator functions refer to different copies. On the other hand,
keeping track of the permutation and indicating by $\sigma^{\alpha}_j(i)$
the  permutation used in such monomial when the coefficient is $c_i$, we have 
that all the random variables related to the same monomial
add up to
\begin{eqnarray*}
&& \sum_{\text{ all permutations}} \sum_{i=1}^m 
c_i \prod_{\alpha} \prod_{j \in J_{\alpha} }
\mathbb I_{\alpha,\sigma^{(\alpha)}_j(i)}
\\
&&  \quad\quad \quad =
\sum_{i=1}^m \sum_{\text{ all permutations}} 
c_i \prod_{\alpha} \prod_{j \in J_{\alpha} }
\mathbb I_{\alpha,\sigma^{(\alpha)}_j(i)}\\
&&  \quad\quad \quad =
\sum_{i=1}^m 
c_i   (\nu - \sum_{\alpha: J_{\alpha} \neq \emptyset} \nu_{\alpha})!
\sum_{\text{  permutations}: J_{\alpha} \neq \emptyset} 
\prod_{\alpha} 
\prod_{j \in J_{\alpha} }
\mathbb I_{\alpha,\sigma^{(\alpha)}_j(i)}\\
&&  \quad\quad \quad = \left(
(\nu - \sum_{\alpha: J_{\alpha} \neq \emptyset} \nu_{\alpha})!
\sum_{\text{  permutations}: J_{\alpha} \neq \emptyset} 
\prod_{\alpha} 
\prod_{j \in J_{\alpha} }
\mathbb I_{\alpha,\sigma^{(\alpha)}_j(1)} \right) \sum_{i=1}^m 
c_i \\
&&  \quad\quad \quad =0
\end{eqnarray*}
where the penultimate equality derives from the fact that in each product all the
terms $\mathbb I_{\alpha,\sigma^{(\alpha)}_j(i)}$ refer to the same number of
different copies by construction, and hence the sum over all permutations
does not depend on $i$.

\bigskip

Therefore, if the payoff of a game is $V$,
the incorrect evaluator is  willing to pay an entry fee to participate,
but the game ends up being a draw all the time. 
This is the weak Dutch Book mentioned in the statement of the theorem.

\end{proof}

Some remarks.  The amount of the entry fee
cannot be predicted in advance but only determined when the
terms in (\ref{polynomialdutchbook2}) are computed,

It is possible to give bounds on the number of copies of the given events, based on 
bounds on the number of polynomials used in Stengle's Theorem \cite{LPR} and
the number of permutations. But such bounds are far from optimal.

The random variable $V$ obtained above, representing the payoff
in the Dutch Book, is often not
the best possible option, especially because 
a large number of permutations has been introduced. For an actual
determination of a game one can often select the copies more
carefully so as to set up a game which is more obviously 
"advantageous" for the incorrect evaluator.
In the Appendix \ref{Example2} there is a very simple example, worked out completely including
an alternative choice for the $V$.

One could observe that it is not obvious that the incorrect evaluator is 
capable of computing $E(V)$ according to his or her own assumptions,
and that exploitation of incorrect probability evaluations is a (possibly deplorable)
art in itself.

\subsection{Variations }

\begin{cor}
If the inadmissible requirements for 
 a finite number of events contain no strict inequalities
 then, assuming the possibility of producing mutually independent, jointly
 perceivable copies of the
 events, a
 strict Dutch Book can be realized against any incorrect evaluator
believing such requirements.

\end{cor}
\begin{proof}
If there are no strict inequalities in the requirements,
since normalization and nonnegativity of probabilities correspond also 
to inequalities which are not strict, there are no strict inequalities in (\ref{system1}).
Hence, we can take $h=1$,
and this generates the multiplicative monoid. From \eqref{polynomialdutchbook1}
we have $F+G=-1$. Following the same construction as in the proof of 
Theorem \ref{DB1}, 
except for the terms in the multiplicative monoid, one gets
$$ V=  \sum_{\text{ all permutations}\{\sigma^{(\alpha  )}\}_{\alpha  \in
\{-1,1\}^{n} }}
\left ( \sum_{r=1}^{m_1} \tilde t_r \tilde f_r +
\sum_{J \subseteq \{m_1+1, \dots, m_1+m_2\}} \tilde s_J \prod_{r \in J} \tilde g_r \right ).
$$
The incorrect evaluator now estimates $E(V)=0$, while for each
realization $V =-1.$

With payoff  $V$, then,
the incorrect evaluator perceives  the game as fair,
while losing a constant unit amount.
This is the Dutch Book mentioned in the statement of the corollary.

\end{proof}
Notice that in this case the amount lost is fixed, and as such known in 
advance of any calculation about the polynomials determining the Dutch Book.

\bigskip

One way to certify inadmissibility of the requirements is to form a Dutch
Book using  some of the $g({\bf x}) \triangleleft 0$'s only. This is the 
case for the example in Appendix \ref{Contradictory}, in which there is a linear subsystem
which has no solutions; in such a case
a simple linear programming technique
leads to a linear combination of the equations certifying the inadmissibility
of the requirements, and a Dutch Book can be formed as in the corollary below,
with just one copy (see also Appendix  \ref{Example3}).

In fact, one can produce an inadmissibility certifying Dutch Book 
by explicitly deriving a contradiction from the requirements and use it 
to construct the Dutch Book.
  We say that a polynomial
 equation or inequality $f({\bf y}) \triangleleft 0$ {\it implies} the inequality $a \geq b$ between
 two polyomials $a({\bf y}), b({\bf y})$ if $a-b= t({\bf y}) f({\bf y})$
 where $t$ is either some polynomial in the variables ${\bf y}$
 for the case in which $ \triangleleft$ is $=$, or
 $t$ is a sum of squares polynomial for the case in which $ \triangleleft$ is $\geq$.
  The next corollary formalizes how we 
  deduce contradictions with probabilistic calculations;
  this is actually a very optimistic description, as we are generally quite limited in deducing contradictions, 
  and are hardly able to use intuition about sum of squares polynomials
  in probabilistic settings.

\begin{cor} \label{Cor4.3}
Given the system  \eqref{system1},
suppose there are polynomials $a_1, \dots, a_n$ in the variables ${\bf y}$
such that 
\begin{enumerate}
\item $a_1 , a_n \in \mathbb R$ (i.e. they do not depend on {\bf y});
\item $a_1 < a_n$;
\item for each $k=1, \dots, n-1$,  $a_k ({\bf y})\geq a_{k+1}({\bf y})$ is implied by
 one of 
the relations in \eqref{system1} with no strict inequality.
\end{enumerate}
Then, assuming that it is possible to realize
a finite but sufficiently large number of independent, jointly perceivable copies of 
the collection of events,
 one can form a Dutch Book as follows.
 Suppose that at step $k$ the relation $f_{i_k}\triangleleft 0$
implies  $a_k \geq a_{k+1}$, and let $t_k$ 
be such that $a_{k+1}({\bf y})-a_k ({\bf y}) = t_k({\bf y}) f_{i_k}({\bf y})$. Then
$v({\bf y})=\frac{1}{a_n-a_1}\sum_{k=1}^{n-1} t_k f_{i_k}=-1$;
and the random variable $V$ obtained as in the proof of
Theorem \ref{DB1} is the payoff a Dutch Book.
\end{cor}
\begin{proof}
Consider the following procedure.
 Start from $a_1$;
if $f_{i_1}  \triangleleft 0$ implies
$a_1 \geq a_2$, then  let $t_1$ be the polynomial such
that $a_1- a_{2}=t_1 f_{i_1} $;
then 
 $a_1= t_1 f_{i_1} + a_{2}$;
continuing one gets
$a_1= t_1 f_{i_1}+t_2 f_{i_2}+a_3 =\dots =\sum_{k=1}^{n-1} t_k f_{i_k} +a_n$.
Thus
$
\sum_{k=1}^{n-1} t_k f_{i_k} = a_1-a_n <0
$
and 
$v({\bf y})= -1.
$
By  replacing the  ${\bf y}$ variables with indicator functions
as in the proof of Theorem \ref{DB1} one gets the payoff random variable
of a Dutch Book, since $t_k f_{i_k} \geq0 $ in all cases, and $a_n > a_1$.
\end{proof}

An example is in Appendix \ref{Example3}.

\section{Continuous variables} \label{Continuous variables}

Probabilities in the continuous case are characterized by the requirements
that $\Omega$ is some Borel  subset of $\mathbb R$ or $\mathbb R^n$,
and  $\mathbb A = \mathcal B_{\Omega}$, the Borel $\sigma$-algebra of $\mathbb R$
or $\mathbb R^n$ restricted to $\Omega$. 

\subsection{Generalized moments problem and nonstandard analysis} \label{MP}

In the classic problem of moments one assigns a Borel 
subset $\Omega$ of $\mathbb R$
or $\mathbb R^n$ and potential moments, and then looks
for existence of random variables (which  are described by
the requirements of being measurable real valued functions)
defined on $\Omega$ and satisfying the prescribed moments 
with respect to the Lebesgue measure  \cite{ST, L10}. As the unknown is the 
distribution of the random variable, this
 is a linear problem.
 
The solution is
generally expressed in terms of Dutch Books, albeit apparently not using this
explicit terminology. In fact, the Riesz-Haviland theorem, which can be used to identify the main
conditions for existence of solutions to moment problems, states that given 
values $m_{\mathbf k},$ $ \mathbf k \in \mathbb N^n$ and a closed  set $\Omega \subseteq 
\mathbb R^n$, there is a probability $\mu$ concentrated on $\Omega$
such that $\int_\Omega {\mathbf x}^{\mathbf k} \mu d{\mathbf x}= m_{\mathbf k}$
if and only if the following happens:
for a polynomial $p({\mathbf x}) = \sum_{{\mathbf k} \in  \mathbb N^n}
a_{\mathbf k} \prod_{i=1}^n x_i^{k_i} $ in the variables ${\mathbf x}=(x_1, \dots, x_n)$
and ${\mathbf m} = \{m_{\mathbf k}\}_{{\mathbf k} \in  \mathbb N^n}$,
we indicate $p({\mathbf m})=\sum_{{\mathbf k} \in  \mathbb N^n}
a_{\mathbf k}m_{\mathbf k}$; then
$p({\mathbf m}) \geq 0$ for every polynomial $p$ such that $p({\mathbf x}) \geq 0$
for all $ {\mathbf x} \in \Omega$.
Suppose now  that an incorrect evaluator of probabilities assigns moments $m_ {\mathbf k} $ to some
 random variables $ {\mathbf X}$ taking values in $\Omega \subseteq \mathbb R^n$,
 but  there is no probability environment for these requirements. Then let $\overline p ({\mathbf x})  $ be the polynomial, nonnegative on $\Omega$, for which the above condition
fails, i.e. $\overline p({\mathbf m})<0$. Then, $V(\mathbf X)=- \overline p( {\mathbf X}) $ 
is a weak Dutch Book against the incorrect evaluator;
in fact, s(he) evaluates $E(V)=E(-\overline p( {\mathbf X}))= -\overline p({\mathbf m})>0$,
while in fact, for every possible value ${\mathbf x}$ of the random variable
$ {\mathbf X}$, $V=- \overline p( {\mathbf x}) \leq 0$. 

This type of results can be easily generalized to situations in which, instead of
giving directly the (potential) moments of a distribution, polynomial
relations between such moments are assigned, provided that the relations involve
compact sets.

\begin{thm} \label{5.1}
 Consider a countable collection of polynomials 
${\mathbf p}=p_i({\mathbf m})$, $i \in \mathbb N$, in the variables
${\mathbf m}= \{m_{\mathbf k} \}_{{\mathbf k}  \in \mathbb N^n}$, and requirements on a 
probability $\mu$ that it is concentrated on a closed 
subset $\Omega \subseteq \mathbb R^n$
and  its moments  $m_{\mathbf k} =\int_\Omega {\mathbf x}^{\mathbf k} \mu d{\mathbf x}$ 
satisfy $p_i({\mathbf m}) \triangleleft_i 0$ for all $i$. If the 
$ \triangleleft_i $'s contain no strict inequalities and there exists
constants $ M_{\mathbf k}, \mathbf k \in \mathbb N^n$,
such that if 
 $\mathbf m$ is a solution of  all the $p_i({\mathbf m}) \triangleleft_i 0$'s  
 then
 \begin{eqnarray}\label{bound}
m_{\mathbf k} \leq M_{\mathbf k} \text{ for each } {\mathbf k}  \in \mathbb N^n,
\end{eqnarray}
 then
there is a probability $\mu$ satisfying the requirements if and only if there
is no Dutch Book against the believer of the above relations.
\end{thm}
In the classic moment problem all $\triangleleft_i$'s are equalities,
hence the relations contain no strict inequalities and 
the $m_{\mathbf k}$'s are bounded; the condition is also fulfilled if, for instance, all equations
 are of the form $a_{\mathbf k} \geq (m_{\mathbf k} - b_k)^2 $, but not if there are some
 $a_{\mathbf k} <m_{\mathbf k}^2$.

\begin{proof}  
 Let $C^{(s)}$ be the set of
$m^{(s)}_{\mathbf k} $ which are solutions
of the relations $p_i({\mathbf m}) \triangleleft_i 0$, $i=1, \dots, s$, and also satisfy
\eqref{bound}; as there are no strict inequalities, the $C^{(s)}$'s form a decreasing 
sequence of compact sets 
in $\mathbb R^{\infty}$ with product topology; and, as \eqref{bound} holds for all solutions 
of all relations $p_i({\mathbf m}) \triangleleft_i 0$, each such solution 
belongs to $C^{(s)}=\cap_{i=1 }^s C^{(i)}$ for each $s$, and hence to
$\cap_{s \in \mathbb N} C^{(s)}$. The absence of any probability satisfying the 
requirements might occur for two reasons:  either
$\cap_{s \in \mathbb N} C^{(s)}= \emptyset$, or for each ${\mathbf m} \in 
\cap_{s \in \mathbb N} C^{(s)}$ there is no solution to the moment problem with 
 ${\mathbf m}=\{m_{\mathbf k}\}_{{\mathbf k} \in \mathbb N}$ as values for the moments.

If  $\cap_{s \in \mathbb N} C^{(s)}= \emptyset$ then there is an $s$ such that 
$ C^{(s)}$ is empty by completeness of $\mathbb R$. Which 
means that necessarily
there is no solution to the polynomial relations $p_i({\mathbf x}) \triangleleft_i 0$, $i=1, \dots, s$. By the Positivstellensatz, there is a polynomial $v({\mathbf x})=-1$
such that  the relations $p_i({\mathbf x}) \triangleleft_i 0$, $i=1, \dots, s$ imply
$v\geq 0$. We ca now perform a construction analogous to the one in the proof of Theorem \ref{DB1},
using enough mutually independent, jointly perceivable copies of the 
random variables on $\Omega$ which have been identified by the incorrect evaluator;
as the expectation factorizes over the product of random variables
depending on mutually independent, jointly perceivable probability spaces,
it is easily seen that we get a strict Dutch Book $V$.

If  $\cap_{s \in \mathbb N} C^{(s)} \neq  \emptyset$, then for each $\widetilde {\mathbf m} \in 
\cap_{s \in \mathbb N} C^{(s)}$  there is no solution to the moment problem with the
 $\widetilde m_{\mathbf k}$'s as values for the moments; hence
 for each for each such $\widetilde {\mathbf m}$ there is a  polynomial $p^{(\widetilde {\mathbf m})}$,
 nonnegative on $\Omega$, for which $ p^{(\widetilde {\mathbf m})}(\widetilde {\mathbf m}) <0$. 
 Each such polynomial determines an open set $ B^{(\widetilde {\mathbf m})}=\{ {\mathbf m}: 
 p^{(\widetilde {\mathbf m})}(  {\mathbf m}) <0\}$; and 
 $$
 \{ B^{(\widetilde {\mathbf m})} \}_{\widetilde {\mathbf m} \in \cap_{s \in \mathbb N} C^{(s)}}
 $$ 
 is an open cover of the compact set $\cap_{s \in \mathbb N} C^{(s)}$, from which we can 
 extract a finite subcover 
 $  \{ B^{(\widetilde {\mathbf m}^{(j)})} \}_{j=1,\dots,r}$. The function
 $v({\mathbf x})=\min_{j=1, \dots, r} p^{(\widetilde {\mathbf m}^{(j)})}(\mathbf x) $ is
nonnegative  on $\Omega$, and yet $v({\mathbf m})<0 $ in each solution 
${\mathbf m} \in \cap_{s \in \mathbb N} C^{(s)}$. Consider payoffs $V=-v({\mathbf  X})$:
then $ \leq 0$
for every realization ${\mathbf  x}$ of ${\mathbf  X}$; on the other
hand, whatever ${\mathbf  m}$ the incorrect evaluator deems the moments to be,
it will be $p^{(\widetilde {\mathbf m})}(  {\mathbf m}) <0$ for some $\widetilde {\mathbf m}$,
so
\begin{eqnarray*}
E(V(X)) &=& E(\max_{j=1, \dots, r} - p^{(\widetilde {\mathbf m}^{(j)})}(\mathbf X) \\
 &\geq &  E (-p^{(\widetilde {\mathbf m})}(  {\mathbf X}) )
 = -p^{(\widetilde {\mathbf m})}(  {\mathbf m}) >0.
 \end{eqnarray*}

Hence, $V$ is a Dutch Book.

\end{proof}

\bigskip

The situation changes if there are strict inequalities in the $p_i({\mathbf m}) \triangleleft_i 0$'s
or the set of solutions is unbounded, as
the set of possible moment values is no longer compact. In such case, absence of
a Dutch Book is compatible with absence of a probability distribution
satisfying the given constraints. In fact, there might be a Loeb
distribution on nonstandard reals \cite{T2012} which satisfies all the requirements.
\begin{ex}
Let $n=1$; $p_0 $ be $ m_1 >0$;
$p_{2 r}$ be $m_1\leq1/r$;  and $p_{2 r+1}
$ be $(m_1)^r-m_r=0$, for $r \in \mathbb N$.
Clearly, there is no standard solution, but the atomic Loeb distribution
concentrated on the hyperreal  
 $(1,1/2, \dots, 1/n, \dots)/\mathcal U  \in {}^*{\mathbb{R}} $,
 where $\mathcal U $ is a fixed ultrafilter,  satisfies
all the requirements.
\end{ex}
It is likely to be the case that in the continuous case
absence of a Dutch Book is equivalent to the existence of
a distribution on nonstandard numbers which satisfies all the requirements.

On the other hand, if one wants to find a structure whose absence guarantees
existence of a standard solution to the
moment problem in general form, i.e. with possible strict inequalities,
one would need to develop a modified
 version of Dutch Books: this 
situation, limited to the existence of a martingale measure in
absence of a modified version of arbitrage, has been solved in \cite{DS08}.

\subsection{A decidable fragment: poly-moments conditions for joint normals} \label{DFP}

Consistency of some collections of requirements is more manageable; in particular, satisfiability may
become decidable in some other  classes of problems besides Elementary Probability.
We briefly present one example below; it is again based on semialgebraic geometry.

\bigskip

 Consider the following requirements. $\Omega = \mathbb R$,
$\mathbb A=  \mathcal B_{\mathbb R}$, and $P=e^{-(\prod_{j=1}^n x_j^2)/2} \lambda$,
where $\lambda$ is the Lebesgue measure; 
moreover, 
there are $n$ random variables
$X_1, \dots, X_n$ satisfying:
\begin{eqnarray}\label{DefJoinGaus}
X_i=\sum_{j=1}^n a_{i,j} Z_j + b_j \text{ for some } a_{i,j}, b_j \in \mathbb R
\end{eqnarray}
where $Z_j$ are such that their joint distribution has density $dP/d\lambda$
(in short, the $Z_j$'s are i.i.d. $N(0,1)$ and the $X_i$'s have joint normal distributions);
 finally,
the $X_i$'s satisfy
 \begin{eqnarray}\label{ReqPolyMomCondJointGauss1}
  \sum_{0 \leq \rho_i \leq s,\text{ for } i=1, \dots, n}
\rho_{k_{1,r}, \dots, k_{n,r}} (r) E(\prod_{i=1}^n X_i^{k_{i,r}}) \triangleleft 0
\end{eqnarray}
for $r=1, \dots, R$, where  the
$\rho_{k_{1,r}, \dots, k_{n,r}} (r)$'s are given real constants, and the $k_{i,r}$'s and $s$ are given integers. 
We call these probability environments poly-moment conditions for joint normals. 
We have

\begin{thm}
For given $\rho_{r}$'s and $k_{i,r}$'s, it is decidable whether the
 requirements for joint normal distributions determine a probability environment or not.
\end{thm}
\begin{proof}
Substitute in  \eqref{ReqPolyMomCondJointGauss1}  
the expressions of $X_i$'s from   \eqref{DefJoinGaus} and expand.
The $Z_j$'s are independent, and their moments are known: $E(Z^k)=
(k-1)!!$ for $k$ even, and $0$ otherwise. Therefore, we get $R$ 
polynomial relations in the variables $a_{i,j}$ and $b_j$, for which the existence
of a real solution is decidable as described before. Once we have 
one selection of values for the $a_{i,j}$'s and the $b_j$'s,  \eqref{DefJoinGaus} gives
the required concrete random variables.

\end{proof}
Notice that, to the opposite of Elementary Probability,  it is not clear under which conditions a
 system of polynomial relations is obtained from
the feasibility test of poly-moment conditions for joint normals.

\titleformat{\section}{\large\bfseries}{\appendixname~\thesection .}{0.5em}{}

\begin{appendices}

\section{A contradictory set up} \label{Contradictory}
We present here a simple exercise taken from a widely
used, application oriented, very high quality textbook. In the exercise, which
is set up so as to mimic a realistic production problem, the
authors propose a set of assumptions about probabilities and ask for
the calculation of several other probabilities; the required calculations
can be carried out without difficulties; 
it is very likely that this exercise has been solved thousands of times.
What is problematic, however,
is that on carrying out one extra 
 calculation one realizes that
the actual set of assumptions is inconsistent. 
In \cite{AT}, Exercise $2.8$ page $67$ presents the following
problem.
\begin{ex} \label{contra2}
On a given day, casting of concrete structural elements at a construction
project depends on the availability of material. The required material
may be produced at the job site or delivered from a premixed concrete 
supplier. However, it is not always certain that these sources of material
will be available. Furthermore, whenever it rains at the site,
casting cannot be performed. On a given day, define the following elements:
\begin{itemize} 
\item[] $E_1=$  there will be no rain
\item[] $E_2=$   production of concrete material at the job site is feasible 
\item[] $E_3=$   supply of premixed concrete is available
 \end{itemize}
 with the following respective probabilities:
$P(E_1)= 0,8$, $P(E_2)= 0,7$, $P(E_3)= 0,95$ and $P(E_3 | E_2^c)= 0,6$
whereas  $E_2$ and $E_3$ are statistically independent of $E_1$.
\begin{itemize} 
\item[(a)] Identify the following events in terms of  $E_1, E_2,$ and $ E_3$:
\begin{itemize}
\item[(i)] $A=$ casting of concrete elements can be performed on a given day; \\
\item[(ii)]  $B=$ casting of concrete elements cannot be performed on a given day.
\end{itemize}
\item[(b)]  determine the probability of the event $B$.
\item[(c)] If production of concrete material at the job site is not feasible, what is 
the probability that casting of concrete elements can still be performed on a 
given day?
 \end{itemize}
 \end{ex}
 We just briefly mention the intended solution. It is implicitly assumed,
 from the theory presented in the book, 
 that all events are jointly perceivable.
\begin{itemize} 
\item[(a)] $A=E_1 \cap(E_2 \cup E_3)$
$B=A^c=E_1^c \cup(E_2^c \cap E_3^c)$;
\item[(b)] by the independence of (any combination of)
 $E_2, E_3$ 
from $E_1$ we have
\begin{eqnarray*}
P(B)&=&1-P(A)=
1-P(E_1 \cap(E_2 \cup E_3))
\\
&=&1-P(E_1) P(E_2 \cup E_3).
\end{eqnarray*}
Since $E_2 \cup E_3= E_2 \cup (E_2^c \cap E_3)$ and
$P(E_2^c \cap E_3)=P( E_3|E_2^c )P(E_2^c)
=0,6  \times(1-0,7)=0,18$, we have
\begin{eqnarray}  \label{calc1}
P(B)
&=&1-(0,8 \times (0,7+0,18))=1-0,704=0.296.
\end{eqnarray}
\item[(c)]
we have
\begin{eqnarray*}
P(A|E_2^c)&=&\frac{P(A \cap E_2^c)}{P(E_2^c)}\\
&=&\frac{P(E_1 \cap(E_2 \cup E_3) \cap E_2^c)}{P(E_2^c)}
\\
&=&\frac{P(E_1 \cap  E_3 \cap E_2^c)}{P(E_2^c)}
\\
&=&\frac{P(E_1)P( E_3 \cap E_2^c)}{P(E_2^c)}
=\frac{0,8 \times 0,18}{0.3}=0.48
\end{eqnarray*}
 \end{itemize}

  Something, however,  is not correct in this set up:
 we have
$P(E_3 \cap E_2^c) = P(E_3|E_2^c)P(E_2^c)
=0.18$ so that
\begin{eqnarray} \label{formula of contradiction}
P(E_3 \cap E_2 )= P(E_3)- P(E_3 \cap E_2^c)=0.77
> 0.7= P( E_2)
 \end{eqnarray}
 which contradicts monotonicity of $P$. Alternatively, again from \eqref{calc1},
 $P(E_3 \cup E_2 )=0.88 < 0.95  = P(E_3)$.

We conclude that the pre-environment described here is contradictory;
the contradiction
  does not appear in the intended calculations, which were all obtained by
 sound applications of inference rules, but only with a careful choice of the
 events to examine.

 As a consequence of the contradiction, for every  statement,
  including  the one in
 \eqref{calc1}, the
 opposite can  also be inferred: as $P(E_2 \cup E_3) \geq P(E_3)=0.95$,
\begin{eqnarray}  \label{calc2}
P(B)
&=&1-(0,8 \times P(E_2 \cup E_3)) \leq 1-0,76=0.24 \neq 0.296.
\end{eqnarray}

\bigskip

\section{Analysis of the contradictory setup} \label{analysis}

The set up described in Appendix \ref{Contradictory} proposes a 
probability environment, which requires admissibility, a decidable question
by Corollary \ref{equiv2}.  As an illustration, we determine once again that the
requirements in  Appendix \ref{Contradictory} are contradictory directly
by the algebraization method of Lemma \ref{equiv2} and Corollary \ref{equiv2}.

First, observe that there are three events involved, so that $\mathbb A= \{E_1, E_2, E_3\}$.
Then select a real variable for each of the boolean combinations 
 on which conditions are given; it is convenient to use the following notation:
 $$
 x_{{\beta_1}, {\beta_2}, {\beta_3}}=
 P(E_1^{\beta_1} \cap E_2^{\beta_2} \cap E_3^{\beta_3})
$$
 where $\beta_m \in \{-1,0,1\}$ and $A^{-1}= A^C,
 A^0=\Omega, A^1=A$.
 The equations expressing the requirements of the probability environment are the
 following, where we have assumed that the claimed independence is 
 actually the full independence
 of the algebra generated by $E_2$ and $E_3$ from the algebra generated by $E_1$ (as it makes sense 
 that the wheather is independent from any combination of human productions):
\begin{eqnarray}
 \begin{cases} \label{eqAngTang}
 x_{1,0,0}& =0.8  \\ 
 x_{0,1,0}& =0.7 \\ 
 x_{0,0,1}& =0.95  \\
 x_{1,1,1}& =x_{1,1,0}\cdot x_{0,0,1}\\ 
 x_{1,-1,1}& =x_{1,-1,0}\cdot x_{0,0,1}\\ 
 x_{-1,1,1}& =x_{-1,1,0}\cdot x_{0,0,1}\\ 
 x_{-1,-1,1}& =x_{-1,-1,0}\cdot x_{0,0,1}\\ 
 x_{0,-1,1}& =0.6 \cdot x_{0,-1,0}
  \end{cases} 
\end{eqnarray}
With the trivial substitution $x_{0,-1,0}= 1 -x_{0,1,0}$ there are $12$ variables in the 
system. 

Now make the change of variables 
$$x_{{\beta_1}, {\beta_2}, \beta_3}
= \sum_{\alpha_m \in \{\beta_m + |\beta_m|-1,\beta_m - |\beta_m|+1\}, 
m=1,2,3}
y_{\alpha_1,\alpha_2,\alpha_3}
$$
where $y_{\alpha_1,\alpha_2,\alpha_3}$ indicates the unknown
probability of $E_1^{\alpha_1} \cap E_2^{\alpha_2} \cap E_3^{\alpha_3}$.
After substitution and the inclusion of the conditions on the $y_{\alpha}$'s the 
system has $9$ equations of degree either $1$ or $2$, and $8$ inequalities.
As discussed here below there is no solution, indicating that the requirements are not admissible; a
certificate that the system has no solution is in Appendix \ref{Example3}.

\bigskip

One can actually wonder if it was the value $0.95$ required for $P(E_3)$ which
created a problem. In fact, one can leave $x_{0,0,1}$ as an indeterminate and solve
the system in $y_{\alpha}$
for the other $7$ variables. The result, before
imposing the  condition that $y_{\alpha}\geq 0$ for all $\alpha$'s, is:
$$
 \begin{cases} \label{eqAngTang}
 y_{-1,1,1}& =(21-50 y_{1,1,1})/50  \\  
 y_{1,-1,1}& =(12-25 y_{1,1,1})/25\\ 
 y_{1,1,-1}& =2 y_{1,1,1}/3\\ 
 y_{-1,-1,1}& =(-3+10 y_{1,1,1})/10\\ 
  y_{-1,1,-1}& =(21-50 y_{1,1,1})/75\\ 
 y_{1,-1,-1}& =(24-50 y_{1,1,1})/75\\
 y_{-1,-1,-1}& =(-3+10 y_{1,1,1})/15.
  \end{cases} 
$$
With the nonegativity condition one has $P(E_1 \cap E_2 \cap E_3)=y_{1,1,1} \in [3/10, 21/50]$,
and $P(E_3)= y_{1,1,1}+y_{-1,1,1}+y_{1,-1,1}+y_{-1,-1,1}=3/5$. This amounts to the
cylindrical decomposition according to Tarsky-Seidenberg reduction theorem, although it is
better computed by directly solving the equations first.

In conclusion, only the value $P(E_3)=3/5$ is admissibile in this set up.


 \section{Example} \label{Example}
 
 simple situation in which our theory applies.
\begin{ex}  \label{boundingwithlessevents}
Show that if five equiprobable collectively independent events are such that 
the probability of each exceeds that of the overall intersection by $0.5$, then
the probability that  at least one event occurs
can be  bounded by the sum of the
probabilities of any three of them (instead of all five as would follow from subadditivity). 
Show also that the probabilities of two of them are not enough in the case above, but they are if
the excess is
$0.55$ instead of $0.5$.
\end{ex}
\begin{incsol}
Let $A_i$, $i=1, \dots, 5$ indicate the five events, and let $a= P(A_i)- P(\cap_{i=1}^5 A_i)$
be the indicated excess. We start from $a=0.5$. Denoting
$P(A_i)=z$,  by inclusion-exclusion, independence
and equiprobability,
the probability that at least one event occurs satisfies
\begin{eqnarray} \label{conditions1}
 p(z,a)&=&P(\cup_{i=1}^5 A_i) \nonumber  \\
&=& \sum P(A_i)- \sum_{i_1\neq i_2} P(A_{i_1} \cap A_{i_2})+
\sum_{i_1\neq i_2 \neq i_3} P(A_{i_1} \cap A_{i_2}\cap A_{i_3}) \nonumber \\
&& \quad \quad -\sum_{i_1\neq i_2 \neq i_3\neq 1_4} P(A_{i_1} \cap A_{i_2}\cap A_{i_3}\cap A_{i_4})+P(\cap_{i=1}^5 A_i)  \\
&=& 5 P(A_1)- 10 P(A_1)^2+10 P(A_1)^3 -5  P(A_1)^4+ P(A_1)^5 \nonumber \\
&=& 6 P(A_1)- 10 P(A_1)^2+10 P(A_1)^3 -5  P(A_1)^4 -a  \nonumber \\
&=& 6 z- 10 z^2+10 z^3 -5  z^4 -a.
\nonumber
\end{eqnarray}
\begin{enumerate}
\item We are then asked to prove that $p(z,0.5)-3z=3z- 10 z^2+10 z^3 -5  z^4 -0.5 \leq 0$
for $z \in [0,1]$, which
is easily shown by simple calculations. In fact, 
$p(z,0.5)-3z \leq  3z- 10 z^2+10 z^3 -3  z^4 -0.5=
z(1-z)(3- 7 z+3 z^2)-0.5 $,
so it is sufficient to show that $q(z)=z(3- 7 z+3 z^2) - 0.5\leq 0$, but
$q(0), q(1) <0$, and $q(z)$ computed at the root of $q'(z)=0$ is negative.

\item On the other hand, $p(0.3, 0.50)= 0.6295  > 2 \times 0.3$, so that $p(z,0.5) \leq 2 z$ does not
hold for all $z \in [0,1]$. Therefore, a bound by the probabilities of two events is not sufficient.

\item Finally, we see that once again $r(z)=p(z, 0.55)-2z <0$ for 
$z \in [0,1]$; in fact,
$r(0), r(1) <0$, $r''(z) <0$ as it is an irreducible second order polynomial, 
and if $z^*$ indicates the only real root  of $r'(z)=0$, computable by solving a third degree polynomial,
then $r(z^*) <0$. So, if the excess $a$ equals $ 0.55$ then the probability of the
union can be bounded by the sum of two of the probabilities of the single events.
\end{enumerate}

\end{incsol}

\begin{impsol}
It is possible to verify the existence of a probability environment 
$(\Omega, \mathbb A, P)$ satisfying 
the requirementsby the algebraization method in  Corollary 
 \ref{equiv2}. Clearly, for this simple case there is plenty of shortcuts, 
 but, by way of exemplification, let's follow the abstract scheme.

There are 
no requirements on $\Omega$ and $\mathbb A$ is required to be of size $5$.
So let $\mathbb A= \{A_i, i=1, \dots, 5\}$, let $a= P(A_i)- P(\cap_{i=1}^5 A_i)$
be the indicated excess, and consider the
boolean combinations $B_{\beta}= \cap_{i \in \beta} A_i$ for
$\beta \subseteq \{1,2\dots, n\}$.
The stated requirements are then expressed by the following system
\begin{eqnarray}  \label{syst1ex1}
\begin{cases} 
x_{\{1,2,3,4,5\}}+a -x_{\{i\}}=0, \quad i=1, \dots, 5 \\
x_{\beta} - \prod_{i \in \beta} x_{\{ i \} }=0, \quad \beta \subseteq \{1,2\dots, n\} \\
x_{\{i \} } - x_{\{j \} }=0, \quad i,j=1, \dots, 5, i \neq j.
\end{cases}
\end{eqnarray}
Perform the change of variables $x_{\beta} = \sum_{\alpha: \alpha_i= 1 \text{ for }
i \in \beta} y_{\alpha}$ and add the normalization and nonnegativity relations to get the system
\begin{eqnarray} \label{syst2ex1}
\begin{cases} 
y_{\{1,1,1,1,1\}}+a -\sum_{\alpha: \alpha_i=1}y_{\alpha}=0, \quad i=1, \dots, 5 \\
\sum_{\alpha: \alpha_i=1 \text{ for }i \in \beta}y_{\alpha} - 
\prod_{i \in \beta} \sum_{\alpha: \alpha_i=1}y_{\alpha}=0, \quad \beta \subseteq \{1,2\dots, n\} \\
\sum_{\alpha: \alpha_i=1}y_{\alpha} - \sum_{\alpha: \alpha_j=1}y_{\alpha}=0,
 \quad i,j=1, \dots, 5, i \neq j \\
 \sum_{\alpha \in \{-1,1\}^5 } y_{\alpha}-1=0.
 
\end{cases}
\end{eqnarray}
The nonemptiness of the semialgebraic set determined by the last system
can be determined by cylindric reduction according to Tarsky-Seidenberg elimination,
or shortened by some substitution. Indicating by $z$
the common value of the $\sum_{\alpha: \alpha_i=1}y_{\alpha}$'s one gets to
\begin{eqnarray}  \label{syst3ex1}
\begin{cases}
z^5+a-z=0\\
\sum_{\alpha: \alpha_i=1 \text{ for }i \in \beta}y_{\alpha} - 
z^{| \beta|} =0, \quad \beta \subseteq \{1,2\dots, n\} \\
 \sum_{\alpha \in \{-1,1\}^n } y_{\alpha}-1=0.
\end{cases}
\end{eqnarray}
For $a=0.5$ one can show by Sturm's theorem,  or similar methods,
that there are indeed solutions in $[0,1]$ (see Remark 
\ref {3.5} below). Assuming $z^*$ is one such solution,
then one can take a concrete probability space with $5$ independent 
events, each with probability $z^*$, which is known to exist; 
in such a space all requirements hold. Therefore, the 
probability environment is well defined and Part $1.$ of the incorrect solution is 
actually
correct.

On the other hand, if  $a=0.55$ then by the same algebraic
methods above, one can see that there is no solution of the first equation
in $[0,1]$ (see Remark 
\ref {3.5} below), and hence the assumptions about the excess being 
$0.55$ are contradictory, and the entire calculation in Part $3.$ of 
the incorrect solution does not make any sense at all.
\end{impsol}
Notice that the the equation $z^5+a-z=0$ cannot be solved by radicals for the 
given values of $a$; hence, the result about the roots is purely existential.

\begin{rem} \label{3.5}
Notice also that Sturm's theorem gives a condition on $a$ for the existence of a solution in $[0,1]$ of
$z^5+a-z=0$. The Sturm sequence in $z=0$ is $a, -1,-a,1-\frac{3125 a^4}{256}$
and in $z=1$ is $a, 4, 4/5 -a, 1-\frac{3125 a^4}{256}$. Hence, there are
two solutions in $[0,1]$ for $a \in [0,\frac{4}{5^{5/4} })$,
one for $a= \frac{4}{5^{5/4} }$ and none for $a > \frac{4}{5^{5/4} }$.
Hence, the requirements are admissible if and only if $a \leq \frac{4}{5^{5/4} } \approx 0.535$.
\end{rem}
\bigskip

There is yet another twist in Part $2.$ of the incorrect solution. There, we are trying to 
derive a negative result from the assumptions, so the
consistency of the assumptions of the probability environment is not
enough; a complete solution of the exercise uses the methods of Corollary \ref{equiv3}

\begin{compsol}(to Example  \ref{boundingwithlessevents}). 
According to the above
corollary, we augment System (\ref{syst1ex1}) by 
the equation $x_{ 1 \vee 2 \vee 3\vee 4 \vee 5} - 2 x_{\{1\} }  > 0$,
where $x_{ 1 \vee 2 \vee 3\vee 4 \vee 5} $ corresponds to 
$P(A_1 \cup A_2 \cup A_3 \cup A_4 \cup A_5)$. After 
the $\bf x $-$ \bf y$ change of variables, and the same substitutions as 
in the improved solution, we get System (\ref{syst3ex1}) 
incremented by the inequality $1-(1-z)^5-2z >0$. Combined
with $z-z^5-0.5=0$ this gives $(1-z)^5/2+z^5 <0$, which is impossible to
satisfy if $z\geq 0$. Since
without the 
additional relation this is a probability environment, it follows from Corollary \ref{equiv3} that 
$P(A_1 \cup A_2 \cup A_3 \cup A_4 \cup A_5) \leq 2 P(A_1)$ holds
 for $a=0.5$, and that also Part $2.$ of the incorrect solution is wrong.

\end{compsol}
The point of the complete solution  to Part $2.$ is that indeed $p(z,0.5) \leq 2 z$
does not hold for all $z \in [0,1]$, but it holds for all the $z$'s for which
the hypothesis make sense.

\section{Example of Dutch Book  } \label{Example2}

\begin{ex}
An incorrect evaluator of the probabilities of two events $A_1$ and
$A_2$ might assume
that: they are independent, 
all their boolean combinations are jointly perceivable,
 $P(A_1|A_2) = 1/2$,
and, finally, that $P(A_1)\neq 1/2$. Setting
$
 x_{{\beta_1}, {\beta_2}}=
 P(A_1^{\beta_1} \cap A_2^{\beta_2})
$ the above equirements  give rise to a system with $4$ terms
 $$
 \begin{cases} \label{syst1indipnonindip}
 x_{1,1}-x_{1,0}x_{0,1}& =0 \\ 
  x_{1,1}-\frac{1}{2} x_{0,1}& =0 \\ 
 x_{0,1}& \neq 0\\ 
 x_{1,0}-\frac{1}{2}& \neq 0.
  \end{cases} 
$$
After substitution with the $y_{\alpha}$'s, one gets a system
with $9$ terms, $m_1=3$ equations, 
$m_2=4$ inequalities, and $m_3=2$ with $\triangleleft$
replaced by $ \neq$.
A polynomial certifying that there is no
solution is
\begin{eqnarray*}
v&=& \sum_{r=1}^{2}  t_r  f_r +
\prod_{r=8}^{9} ( h_r )^{2}
\\ &=&\left((y_{1,1}+y_{-1,1})(y_{1,1}+
y_{1,-1}-\frac{1}{2})\right) 
 \left(y_{1,1} -\frac{}{}(y_{1,1}+y_{1,-1})(y_{1,1}+y_{-1,1})\right)
\\ && \quad 
+
\left(-(y_{1,1}+y_{-1,1})(y_{1,1}+
y_{1,-1}-\frac{1}{2})\right)
 \left(y_{1,1} -\frac{1}{2}(y_{1,1}+y_{-1,1})\right)
\\ &&  \quad +(y_{1,1}+y_{-1,1})^2 (y_{1,1}+y_{1,-1}-\frac{1}{2})^2 \equiv 0.
\end{eqnarray*}
Now following steps 1.-6. in the proof of the Theorem \ref{DB1}
for a fixed set of permutations, we get 
\begin{eqnarray*}
\tilde t_2 \tilde f_2&=&
\left(- \mathbb I_{1,1,(\sigma^{(1,1 )}_1)}  \mathbb I_{1,1,(\sigma^{(1,1 )}_2)}
- \mathbb I_{1,1,(\sigma^{(1,1)}_1)}  \mathbb I_{1,-1,(\sigma^{(1,-1 )}_5)}
+\frac{1}{2} \mathbb I_{1,1,(\sigma^{(1,1 )}_1)}\right.
\\ && \quad \quad \left.
-\mathbb I_{1,1,(\sigma^{(1,1 )}_1)}  \mathbb I_{-1,1,(\sigma^{(-1,1)}_7)}
-\mathbb I_{-1,1,(\sigma^{(-1,1 )}_7)}  \mathbb I_{1,-1,(\sigma^{(1,-1)}_5)}
+\frac{1}{2} \mathbb I_{-1,1,(\sigma^{(-1,1 )}_7)}\right)
\\ && \quad
\cdot
 \left (\frac{1}{2} \mathbb I_{1,1,(\sigma^{(1,1)}_3)}
-\frac{1}{2} \mathbb I_{-1,1,(\sigma^{(-1,1 )}_8)}\right)
\end{eqnarray*}
and $\tilde u$ and $V$ accordingly. One can see that there is no
easy way of controlling the different copies  which are used, except
that of taking all permutations.

\medskip

On the other hand, one can construct a simpler and more direct payoff by 
identifying the events as much as possible with the original 
$A_i$'s, rather than with their expansion in standard normal form, and by a
 clever choice of the copies. Here is a possible form  (in which 
 the number of the copy is indicated in the indices):
\begin{eqnarray*}
V'&=&
\left( \mathbb I_{A_{2},(1) } ( \mathbb I_{A_{1},(3)} -\frac{1}{2}) \right)
\left( \mathbb I_{A_{1}\cap A_2,(2)} -\mathbb I_{A_{1},(4)} \frac{}{}
\mathbb I_{A_{2},(2)}\right)
\\ && \quad 
+ \left( -\mathbb I_{A_{2},(1) } ( \mathbb I_{A_{1},(3)} -\frac{1}{2}) \right)
\left( \mathbb I_{A_{1}\cap A_2,(2)} -\frac{1}{2}\mathbb I_{A_{2},(2)}\right)
\\ && \quad 
+\mathbb I_{A_{2},(1)}
  ( \mathbb I_{A_{1},(3)} -\frac{1}{2})( \mathbb I_{A_{1},(4)} -\frac{1}{2}).
\end{eqnarray*}
Since in each factor of each product there appear only events belonging to copies
which are different from those appearing in the other factors
(of the same product),
independence can be used to show that $E(V')>0$ for 
the incorrect evaluator; on the other hand, $V'\equiv 0$ as checked by
simple algebraic expansion.

\end{ex}

\section{A short cut Dutch Book} \label{Example3}

\begin{ex} \label{DutchBook for main example}
According to Corollary \ref{Cor4.3}, it is possible  to use a known contradiction
to build a Dutch Book. We do it for the 
 contradiction found at the end of Appendix \ref{Contradictory} for the problem presented there,
 using the formalization
in Appendix \ref{analysis}.
The chain of deductions leading to the contradiction in \eqref{formula of contradiction},
in the notation of Appendix \ref{analysis} 
is the following: 
\begin{eqnarray*}
x_{0,1,1}-(x_{0,0,1}-x_{0,-1,1})=0 \\
x_{0,0,1}-0.95=0 \\
x_{0,-1,1}-0.6 x_{0,-1,0}=0\\
x_{0,-1,0}-(1-x_{0,1,0})=0\\
x_{0,1,0}-0.7=0\\
x_{0,1,0} - x_{0,1,1} \geq 0\\
x_{0,1,0}-0.7=0
\end{eqnarray*}
Notice that the fifth and the seventh equations coincide, as this relation appears twice in 
the reasoning expressing the contradiction. Now, we combine these relations
with suitable coefficients
in such a way that the single variables are telescopically canceled; 
we can use any real coefficient for all relations except the 
sixth, which needs a nonnegative coefficient. We 
indicate directly the random variables, as all relations are linear and there
is no need of using copies; moreover, we directly replace the ${\bf x}$ variables by
a random variable, as we already expressed the properties of probability
in equations number $1,4$ and $6$. $\mathbb I_{i,j,k}$ is the random variable substituting
the variable $x_{i,j,k}$. The payoff of the Dutch Book is
\begin{eqnarray*}
\tilde v=&& \frac{1}{0.07}[
(\mathbb I_{0,1,1}-(\mathbb I_{0,0,1}-\mathbb I_{0,-1,1})) + (\mathbb I_{0,0,1}-0.95)\\
&& -(\mathbb I_{0,-1,1}-0.6 \mathbb I_{0,-1,0})- 0.6(\mathbb I_{0,-1,0}-(1-\mathbb I_{0,1,0}))\\
&&
+0.6(\mathbb I_{0,1,0}-0.7)+(\mathbb I_{0,1,0} - \mathbb I_{0,1,1}) - (\mathbb I_{0,1,0}-0.7)=-1.
\end{eqnarray*}

\end{ex}

\bigskip

\section{Paradoxes}  \label{paradoxes}
We make a brief reference  to the so called paradoxes,  many
of which have been proposed concerning the foundations of probability
theory
(see, for instance, \cite{E2012, Ha}). 
The point in most of the paradoxes in probability is that  one ether sets up
contradictory collection of assumptions, or is forced by too rigid axiom systems
to assign probabilities where there is no natural way of doing it. Both
problems are addressed and solved by probability environments.

As just one example of how our formulation can deal with paradoxical settings, 
we consider Humphrey's paradox \cite{Hu} and its variants \cite{Lyon(2014)},
which turn out to be troubling for most foundations of probability. 
The paradox gives reasons to make assumptions about some events 
$B_{t_0}, I_{t_1}, T_{t_2}$, namely
\begin{itemize}
\item [(i)] $P (T_{t_2}|I_{t_1} \cap B_{t_0} ) = p > 0,$
\item [(ii)] $1 > P (I_{t_1} | B_{t_0} )  = q > 0,$
\item [(iii)] $P (T_{t_2}| \neg I_{t_1} \cap B_{t_0} )  = 0$;
\end{itemize}
each variant gives then reasons to make an assumption about
 $P (I_{t_1} |T_{t_2} \cap B_{t_0} )$,
for instance that it equals $P (I_{t_1} | \neg T_{t_2} \cap B_{t_0} )
=P (I_{t_1} |  B_{t_0} )$ in the first variant.
It can be easily seen that the statements of the first three variants of the paradox 
are simply setting forth an inconsistent collection of requirements,
 easily detectable by algebraization. In each case,  a Dutch Book can be produced
 against the believer of such assumptions: in the first variant, for instance,
 it can be obtained as follows. Let $\mathbb I_{A,j}, j=1,2$ be
 the indicator function of the event $A$ in copy $j$, for  two independent copies
 of the events. Then, by omitting the (mathematically irrelevant) indications of 
 the times $t_0, t_1, t_2$, let
 \begin{eqnarray*}
 V&=&\mathbb I_{B,1}  \mathbb I_{T,2} \mathbb I_{B,2} (1-q)
 +( \mathbb I_{T,2}  \mathbb I_{B,2}  \mathbb I_{I,1} \mathbb I_{B,1}
 - \mathbb I_{B,1}  \mathbb I_{T,2}  \mathbb I_{I,2} \mathbb I_{B,2})\\
 && \quad +( \mathbb I_{T,2}  \mathbb I_{B,2} q \mathbb I_{B,1}
 -  \mathbb I_{T,2}  \mathbb I_{B,2}  \mathbb I_{I,1} \mathbb I_{B,1})
 -   \mathbb I_{B,1} \mathbb I_{T,2} (1- \mathbb I_{I,2}) \mathbb I_{B,2} ;
  \end{eqnarray*}
  we have
   \begin{eqnarray*}
 E(V)&=&P(B) P(T \cap B) (1-q)\\
 &&
 +( P(T\cap B) P(I \cap B)
 - P(B) P(T \cap I \cap B) )\\
 && +( P(T \cap B) q P(B)
 - P(T \cap B) P(I\cap B)  )
 -  P(B) P(T \cap I^c \cap B)\\
 &=&P(B) P(T \cap B) (1-q)\\
 &&
 +( P(T\cap B) P(I \cap B)
 - P(B) P(T  \cap B) )\\
 && +( P(T \cap B) q P(B)
 - P(T \cap B) P(I\cap B)  ).
  \end{eqnarray*}
  From the requirements we know
  $P(T \cap B) \geq P(T \cap I \cap B) = p P(I \cap B) >0$
  and $P(B)>0$; the first term is strictly positive, and all the other terms in the r.h.s. of the last
  equation are $0$; hence from the requirements, $E(V) > 0$.
  On the other hand, a simple expansion shows that $V \equiv 0$. Therefore,
  $V$ is a weak Dutch Book, witnessing the inconsistency of the
  requirements.

 The fourth variant of the paradox
suggests to leave  $P (I_{t_1} |T_{t_2} \cap B_{t_0} )$ as "undefined". 
We can translate this assumption into the setting of probability environments by 
stating that no requirement is made on such probability. In such case,
the other requirements are consistent, and the paradoxical nature
of the example disappears; a probability environment exists,
and  the last probability can then be computed
as necessary consequence of the other requirements in such environment,
necessarily taking the value $P (I_{t_1} |T_{t_2} \cap B_{t_0} )=1$
(as easily seen by algebra or by Bayes formula).

All of the most relevant formalizations of Probability Theory seem to be affected 
by Humphrey's paradox, as they force us to assign a value to $P (I_{t_1} |T_{t_2} \cap B_{t_0} )$
before hand, hence entering in one of the three contradictory variants. Only
R\'enyi's axiom system allows to leave such probability as undefined, 
and hence is able to "respond" to the paradox,  reaching the same conclusion
as with probability environments.

In a further extension of the paradox (see \cite{Lyon(2014)} Section 6.6), though, 
one can set up things in such a way that also R\'enyi's axiom system
would be forced to assign values which should remain undefined;
the issue is immaterial for probability environments, which then solve 
extension of the paradox  as well.

.

\end{appendices}

\section*{Aknowledgments}
The author would like to thank P. Dai Pra, M. Maggesi, R. W. J. Meester, D. Mundici and C.C.A. Sastri for valuable discussions and comments.

\end{document}